\documentclass{amsart}

\usepackage{amssymb}
\usepackage{amsrefs}

\newtheorem{theorem}{Theorem}[section]
\newtheorem{lemma}[theorem]{Lemma}
\newtheorem{proposition}[theorem]{Proposition}
\newtheorem{corollary}[theorem]{Corollary}
\newtheorem{fact}[theorem]{Fact}
\newtheorem{remark}[theorem]{Remark}
\theoremstyle{definition}
\newtheorem{definition}[theorem]{Definition}

\newcommand{\ep}{\varepsilon}
\newcommand{\vf}{\varphi}
\newcommand{\cv}{\mathrm{cv}}
\newcommand{\conv}{\mathrm{conv}}
\newcommand{\dist}{\mathrm{dist}}

\newcommand{\R}{\mathbb{R}}
\newcommand{\N}{\mathbb{N}}

\newcommand{\rd}{{\mathbb R}^d}
\newcommand{\diam}{{\rm diam}\,}
\newcommand{\Crit}{{\rm Crit}}
\newcommand{\BT}{{BT_{\frac12}}}
\newcommand{\GB}{\cal G_B}
\newcommand{\card}{{\rm card}}

\def \cal{\mathcal}

\begin{document}                                                 
    \title[Critical values of distance functions]{Smallness of the set of critical values of distance functions in two-dimensional
 Euclidean and Riemannian spaces}                                 
    \author{Jan Rataj}
    \author{Lud\v ek Zaj\'\i\v cek}
    \address{Charles University, Faculty of Mathematics and Physics, Sokolovsk\'a 83,
186 75 Praha 8, Czech Republic}
    \email{rataj@karlin.mff.cuni.cz}
    \email{zajicek@karlin.mff.cuni.cz}
		\thanks{The research was supported by the Czech Science Foundation, Project No.~18-11058S.}
		\keywords{distance function, distance sphere, critical point, critical value, Minkowski content, Riemannian manifold}
		\subjclass{57N40,  53B21}
		\begin{abstract}
We study how small is the set of critical values of the distance function from a compact (resp. closed) set in the plane or in a connected complete two-dimensional Riemannian manifold. We show that for a compact set, the set of critical values is compact
and Lebesgue null (which is a known result) and that it has ``locally'' (away from $0$) bounded sum of square roots of lengths of gaps (components of the complement). In the planar case, these conditions of local smallness are shown to be optimal. These  results improve and generalize those of Fu (1985) and of our earlier paper from 2012. We also find an optimal condition for the smallness of the whole set of critical values of a planar compact set.
    \end{abstract}
				           
\maketitle

\section{Introduction}
If $X$ is a metric space and $\emptyset \neq F \subset X$ a closed set, we denote by 
$$d_F:=\dist(\cdot,F)$$ 
the \emph{distance function} to $F$ and by
$$S_r(F):= \{x \in X:\ d_F(x)=r\},\quad r>0,$$ 
the \emph{distance spheres} (called also $r$-boundaries) of $F$ (see \cite{Fe}). Distance spheres and their properties (in $\R^d$ and some more general spaces)  were investigated in a number of articles (see, e.g,  \cite{Fe}, \cite{Fu}, \cite{RZ1}, \cite{RZ2}, \cite{BMO}; for other references see \cite{RZ1}). 

If $X=\R^d$, $d=2,3$, then  for almost all $r>0$, $S_r(F)$ is either empty or a Lipschitz $(d-1)$-dimensional manifold, see Fu \cite{Fu} (where it is factually proved that it is a ``locally semiconcave surface''; for details see \cite{RZ1}).
Moreover, Fu proved in \cite{Fu} that ``for almost all $r>0$'' can be improved in $\R^2$. 
 To describe shortly these results, we introduce the following notation.

For a closed $\emptyset \neq F\subset \R^d$, we denote by $\cv(d_F)$ the set of all \emph{critical values} of $d_F$ (see Definition \ref{critreg} below) and by $T_F$ ($L_F$) the set of all
 $r>0$ for which $S_r$ is nonempty and it is not a topological (Lipschitz, resp.) $(d-1)$-dimensional manifold. A well-known fact is that
$T_F\subset L_F\subset\cv(d_F)$ (cf.\ \eqref{tlcv}). Hence, any result on the smallness of $\cv(d_F)$ implies the corresponding result on the smallness of $L_F$ and $T_F$.

Fu (factually) proved that if $\emptyset \neq F\subset \R^2$ is closed then $\cal \cal H^{1/2}(\cv(d_F))=0$. We showed in earlier papers that the same is true in two-dimensional Riemannian manifolds, Alexandrov spaces and some two-dimensional Banach (=Minkowski) spaces \cite{RZ1} and on two-dimensional convex surfaces \cite{RZ2}.

Another result of Fu in \cite{Fu} was that if $\emptyset \neq F \subset \R^2$ is compact and $\ep>0$, then  $\cv(d_F) \cap [\ep,\infty)$ is a compact set of entropy (= upper Minkowski) dimension at most $\frac 12$.  

In order to formulate our results, we introduce the following notation (see Definition~\ref{bano}). Given a compact set $K\subset\R$ and $\alpha>0$, the \emph{degree}-$\alpha$ \emph{gap sum} of $K$ is defined as $G_\alpha(K):=\sum_{I\in{\mathcal G}_K}|I|^\alpha$, where ${\mathcal G}_K$ is the set of all bounded components of $\R\setminus K$. We say that $K$ is a \emph{$BT_\alpha$-set} if it is Lebesgue null and $G_\alpha(K)<\infty$. Note that each $BT_\alpha$-set has zero $\alpha$-dimensional Minkowski content (see \eqref{nulmin}).

The following is an improvement of the above mentioned Fu's result. In fact, it is a characterization of the smallness of $\cv(d_F) \cap [\ep,\infty)$ for $\emptyset \neq F\subset\R^2$ compact.

\begin{theorem}\label{odrn}
Let $\ep>0$ and $A \subset [\ep, \infty)$. Then the following properties are equivalent.
\begin{enumerate}
\item[(i)]\ $A\subset L_F$ for some compact $F \subset \R^2$.  
\item[(ii)]\ $A\subset \cv(d_F)$ for some compact $F \subset \R^2$.  
\item[(iii)]\ $\overline{A}$ is a $\BT$ set.
\end{enumerate}
\end{theorem}

The main tools from the proof are results on the critical points of DC functions, an inequality due to Ferry (see \eqref{ferryn}) and a construction based again on the Ferry's paper \cite{Fe} (see Proposition~\ref{konstr}). Up to the Ferry's inequality, our approach is completely independent of \cite{Fu} and, hence, provides also an alternative proof of the two-dimensional results of Fu \cite{Fu}.

Using similar methods, we also obtain a result on the smallness of the set of critical points of $d_F$ (Proposition~\ref{mkrb}).

Our second main result concerns the smallness of the whole set of critical values $\cv(F)$ for a compact planar set $F$. The degree-$\frac 12$ gap sum of $\cv(F)$ can be infinite, but, a more careful quantitative local study of the degree-$\frac 12$ gap sum of $\cv(F)$ (Lemma~\ref{nakouli}) makes it possible to obtain the following (again optimal) result:

\begin{theorem}\label{chardon}
Let $A\subset (0,\infty)$. Then the following conditions are equivalent.
\begin{enumerate}
\item[(i)] $A\subset  L_F$ for some compact set $F \subset \R^2$.
\item[(ii)] $A\subset \cv(d_F)$ for some compact set $F \subset \R^2$.
\item[(iii)]  $\overline{A}$ is bounded, Lebesgue null and  
$$\int_0^\infty G_{1/2}(\overline{A}\cap[r,\infty)) \sqrt{r}\, dr<\infty.$$
\end{enumerate}
\end{theorem}

As a corollary, we obtain that the $\frac 45$-dimensional Minkowski content of $\cv(d_F)$ vanishes for nonempty compact planar sets $F$ (but the degree-$\frac 45$ sum of $\cv(d_F)$ can be infinite, in which case the upper Minkowski dimension of $\cv(d_F)$ equals $\frac 45$), see Theorem~\ref{45}.

For $T_F$, we do not know a characterization of smallness, but we present an example of a compact planar set $F$ for which the Hausdorff dimension of $T_F\cap[1,2]$ equals $\frac 12$ (Theorem~\ref{TF}).

We also consider a two-dimensional complete Riemannian manifold $M$. If $\emptyset \neq F\subset M$ is closed, the critical values of $d_F$ are defined as those of $d_F\circ \varphi$ for any chart $\varphi$ of $M$ (see Definition~\ref{D_crit}) and the sets $\cv(d_F)$, $L_F$ and $T_F$ are defined exactly as in the Euclidean case. Again, the inequalities $T_F\subset L_F\subset\cv(d_F)$ hold (see \eqref{tlcv_riem}). With the help of a generalized Ferry's inequality (Proposition~\ref{P1}) we show that a slightly weaker implication than (ii)$\implies$(iii) from Theorem~\ref{odrn} is still true:

\begin{theorem}\label{proko}
 If $X$ is a connected complete two-dimensional Riemmanian manifold, $\emptyset \neq F\subset X$ a compact set, and  $0<\ep <K$,
 then $\cv(d_F) \cap [\ep,K]$ is a  $\BT$ set. In particular, both  $\cv(d_F) \cap [\ep,K]$ and $L_F \cap [\ep,K]$
 have zero $1/2$-dimensional  Minkowski content.
\end{theorem}

We also present related results on the smallness of the sets of critical values for the distance function from closed (not necessarily compact) subsets of $\R^2$ (Theorem~\ref{euprouz}) and of a two-dimensional complete Riemannian manifold (Theorem~\ref{prouz}).

In the Preliminaries, we collect and prove some results about the gap sums of compact subsets of $\R$  (Subsection~2.2) and critical points of Lipschitz and DC functions (Subsection~2.3). Section~3 contains our main results in $\R^2$, and in Section~4 we treat the two-dimensional Riemannian manifolds.

\section{Preliminaries}

\subsection{Basic definitions}
The integer part of $x \in \R$ is denoted by $\lfloor x \rfloor$.
In any vector space $V$, we use the symbol $0$ for the zero element. 
The symbol $B(x,r)$ (resp.\ $\overline B(x,r)$) will denote the open (resp.\ closed) ball
  with center $x$ and radius $r$. The boundary of a set $M$ is denoted by $\partial M$.

In the Euclidean space $\R^d$  the norm is denoted by $\|\cdot\|$ and the  scalar product  by $\langle \cdot,\cdot\rangle$. We denote $S_{\R^d}:= \{x\in \R^d:\ \|x\|=1\}$. The Lebesgue measure in $\R^d$ is denoted by $\lambda_d$. If $I\subset \R$ is an interval, we set $|I|:= \lambda_1(I)$.

A mapping is called $K$-Lipschitz if it is Lipschitz with a (not necessarily minimal) constant $K$.
A bijection $f$ is called bilipschitz if both $f$ and $f^{-1}$ are Lipschitz. 

We say that a metric space $X$ is a {\it $k$-dimensional topological (resp. Lipschitz) manifold}
 if for every $a \in X$ there exists an open neighbourhood $U$ of $a$ and a homeomorphism
 (resp. a bilipschitz  homeomorphism) of $U$ on an open subset of $\R^k$.

We will need also the following special notation.
\begin{definition}\label{lipgr}
Let $A\subset \R^2$. We will say that $A$ is a {\it Lipschitz graph} if there exists a 
Lipschitz function $g:\R \to \R$ such that
$$  A= \{ (x,g(x)):\ x \in \R \}\quad \text{or}\quad A= \{ (g(y), y):\ y \in \R \}.$$
 If $g$ is $M$-Lipschitz, we say that $A$ is an {\it $M$-Lipschitz graph}.
\end{definition}

If $f$ is a real function defined on an open set $G \subset \R^d$, then
the {\it directional derivative} and the {\it one-sided  directional derivative} 
    of $f$ at $x\in G$ in the direction $v \in \R^d$ are defined by 
   $$ f'(x,v) := \lim_{t \to 0} \frac{f(x+tv)-f(x)}{t} \quad \text{and}\quad f'_+(x,v) := \lim_{t \to 0+} \frac{f(x+tv)-f(x)}{t}.$$

	Let $f$ be a real function defined on an open convex set $G \subset \R^d$. Then we say
	 that $f$ is a {\it DC function}, if it is the difference of two convex functions. Special DC
	 functions are semiconvex and semiconcave functions. Namely, $f$ is a {\it semiconvex} (resp.
	 {\it semiconcave}) function, if  there exist $a>0$ and a convex function $g$ on $G$ such that
	$$   f(x)= g(x)- a \|x\|^2\quad (\text{resp.}\ \  f(x)= a \|x\|^2 - g(x)),\quad x \in G.$$
	
	 For $s\ge 0$, denote (following \cite{Ma}) by $\cal H^{s}$ the (nonnormalized) $s$-dimensional Hausdorff measure, and $\dim_H A$ denotes the Hausdorff dimension of $A$.

	Let
	 $d_A:= \dist(\cdot,A)$ be the  distance function from the set $A$, and 
	\begin{equation}\label{para}
A_{\ep}:=\{ z\in\rd:\, d_A(z)\leq \ep\}
\end{equation}
 be
 its $\ep$-parallel set.

 The \emph{$s$-dimensional lower and upper Minkowski content} of a nonempty bounded set $A\subset\R^d$ is defined by
\[
\underline{\cal M}^s(A):=\liminf_{\ep\to 0+} \frac{\lambda_d(A_{\ep})}{\ep^{d-s}} \quad \text{ and } \quad
\overline{\cal M}^s(A):=\limsup_{\ep\to 0+} \frac{ \lambda_d(A_{\ep})  }{\ep^{d-s}}.
\]
(Note that other definitions in the literature differ by normalization factors only, and so  our results
 hold also under these definitions.)

If $\underline{\cal M}^s(A)=\overline{\cal M}^s(A)$, then the common value ${\cal M}^s(A)$ is refered to as the \emph{$s$-dimensional Minkowski content} of $A$. 
We denote by
\[
\underline{\dim}_M A:=\inf\{t\ge 0 : \underline{\cal M}^s(A)=0\}=\sup\{t\ge 0 :\underline{\cal M}^s(A)=\infty\}
\]
and
\[
\overline{\dim}_M A:=\inf\{t\ge 0 :\overline{\cal M}^s(A)=0\}=\sup\{t\ge 0 :\overline{\cal M}^s(A)=\infty\}
\]
the \emph{lower} and \emph{upper Minkowski dimension} of $A$.	

The \emph{packing dimension} of a set $A\subset\R^d$ can be defined by
\begin{equation}\label{dimp}
\dim_PA:=\inf\left\{ \sup_i\overline{\dim}_MA_i:\, A\subset\bigcup_{i=1}^\infty A_i\right\}
\end{equation}
(cf.\ \cite{Fa}*{Proposition~3.8}).
Recall that (see \cite{Ma}*{p.~79})
\begin{equation}\label{nuldim}
\cal{H}^s(A)=0\ \ \ \text{whenever}\ \ \ \cal{M}^s(A)=0
\end{equation}
and that (see \cite{Fa}*{Eq.~(3.29)})
\begin{equation}   \label{nedi}
\dim_H(A) \leq \dim_PA \leq \overline{\dim}_M A.
\end{equation}

Note also that, in contrast to the (upper, lower) Minkow\-ski dimension, the packing dimension is already stable with respect to countable unions, i.e.,
$$\dim_P\left(\bigcup_iA_i\right)=\sup_i\dim_PA_i,\quad A_i\subset\R^d,\quad i=1,2,\dots$$
(see \cite{Fa}*{Eq.~(3.26)}).

	\subsection{Gap sums}

\begin{definition}\label{bano}
If  $B\subset \R$ is compact, by a \emph{ gap of $B$} we mean a bounded component of $\R \setminus B$
 and by $\cal G_B$ we denote the collection (possibly empty) of all gaps of $B$. 

If $\alpha> 0$, we define the \emph{degree-$\alpha$ gap sum} of $B$ as
$$  G_{\alpha} (B) = \sum_{I \in \GB}   |I|^{\alpha}.$$  

We will say that $B\subset \R$ is a \emph{$BT_{\alpha}$ set} if $B$ is  (possibly empty) compact,
 $\lambda(B) =0$ and  $G_{\alpha} (B) < \infty$.
\end{definition}

 The notations  $\cal G_B$ and $  G_{\alpha} (B)$ are taken from \cite{BN}.
 $BT_{\alpha}$ sets were (factually) first considered by  Besicovitch and Taylor in
 \cite{BT} and were used in \cite{Ka} and \cite{BN} in their study of optimal versions of Sard theorems
 for real $C^k$ and  $C^{k,s}$  functions of one variable.

There exist close connections between gap sums and upper Minkowski content and dimension.
 In \cite{Ha}, it is proved, that if $B\subset \R$ is a compact Lebesgue null set, then
\begin{equation}\label{hawk}
\overline{\dim}_M B  = \inf\{\alpha>0:\  G_{\alpha} (B)< \infty\}=: i(B).
\end{equation}
(This is the correct formulation of \cite{Ha}*{Theorem 3.1}, where the assumption of nullness of
 $B$ is forgotten but used in the proof.)

 The equality \eqref{hawk} was  proved independently
 in \cite{BN}; it is a part of \cite{BN}*{Theorem 1.2}, which contains also (as the implication  $(4) \Rightarrow (5)$) the following
 result.
\begin{equation}\label{nulmin}
 \text{If $B\subset \R$ is a $BT_{\alpha}$ set, then  ${\cal M}^{\alpha}(B)= 0$.}
\end{equation}
Note that the proof of \eqref{hawk} in \cite{Ha} is rather laconic and the proof of 
 \cite{BN}{Theorem~1.2} is rather indirect and  not detailed. So, since \eqref{hawk} and \eqref{nulmin}
 are important for us, we present for completeness their short detailed proofs.
\smallskip

 Since the case of a finite $B$ is trivial both in \eqref{hawk} and \eqref{nulmin}, we consider
further an infinite compact Lebesgue null set $B\subset \R$. Then the collection  $\cal G_B$
 of all gaps is infinite and their lengths may be arranged as a non-increasing sequence
 $(a_n), \ n=1,2,\dots$. For each  $0<r< a_1/2$, let $i=i(r)$ be such that
\begin{equation}\label{ri}
  a_{i+1} \leq 2r < a_i.
	\end{equation}
	It is easy to see (cf.\  \cite{Fa2}*{(3.17), p.~51}) that
	\begin{equation}\label{expl}
	 \lambda_1(B_r)= 2r + 2ri + \sum_{j=i+1}^{\infty} \, a_j,
	\end{equation}
  where $B_r$ is defined in \eqref{para}.
	If $B$ is as in \eqref{nulmin}, the sequence $((a_i)^{\alpha})$ is non-increasing
	 and $\sum (a_i)^{\alpha}$ converges; consequently by a well-known easy fact 
	(see, e.g., \cite{Br}*{p.~31}) we obtain
	 $i\cdot (a_i)^{\alpha} \to 0$. So, since  by  $ \eqref{ri}$ 
	$$   \frac{ir}{r^{1-\alpha}} = i r^{\alpha}  \leq  \frac{i\cdot (a_i)^{\alpha} }{2^{\alpha}},$$
	 we easily obtain
	\begin{equation}\label{prcl}
	\lim_{r\to 0+} \frac{ 2r + 2ir}{r^{1-\alpha}} =0.
	\end{equation}
		Further by \eqref{ri}
		$$
		D(r):= \frac{\sum_{j=i+1}^{\infty} \, a_j}{r^{1-\alpha}} \leq  2^{1-\alpha} \sum_{j=i+1}^{\infty} 
		 \frac{a_j}{(a_{i+1})^{1-\alpha}}
		\leq  2^{1-\alpha}\sum_{j=i+1}^{\infty} (a_j)^{\alpha},
$$
 and consequently  $\lim_{r\to 0+} D(r) =0$. Hence, using \eqref{expl} and \eqref{prcl}, we obtain
 that $\lim_{r\to 0+}  \lambda_1(B_r)/r^{1-\alpha} =0$ and so ${\cal M}^{\alpha}(B)= 0$.

Thus \eqref{nulmin} is proved, which immediately implies that $\overline{\dim}_M B 
\leq i(B)$. So, to prove \eqref{hawk}, it suffices to prove the opposite implication.
 To this end, consider arbitrary  $\alpha$, $\beta$ with $\overline{\dim}_M B < \alpha < \beta$.  Then $\overline{\cal M}^{\alpha}(B)= 0$
 and thus \eqref{expl} gives that  $\lim_{r\to 0+}  ir/r^{1-\alpha} = \lim_{r\to 0+} i r^{\alpha}  =0$
 and, consequently, \eqref{ri} implies that  $\lim_{i \to \infty} i  2^{-\alpha} (a_{i+1})^\alpha  = 0$.
 Thus, for all sufficiently large $i$, we have 
$$(a_{i+1})^\alpha \leq \frac{1}{i}\quad \text{and so}\quad (a_{i+1})^\beta \leq \frac{1}{i^{\beta/\alpha}},$$
and consequently  $i(B)\leq \beta$. By the choice of $\alpha$ and $\beta$, we obtain
 $i(B) \leq  \overline{\dim}_M B$.

\medskip
We will need also the following easy facts on gap sums.

\begin{lemma}\label{mongap}
Let  $\emptyset \neq K_1 \subset K_2 \subset \R$ be compact Lebesgue null sets and $0< \alpha<1$. Then   $G_{\alpha} (K_1) \leq G_{\alpha} (K_2)$.
 \end{lemma}
\begin{proof}
For each  $J \in \cal G_{K_1}$, set  $\cal G^J:= \cal G_{K_2 \cap \overline{J}}$. Then clearly
 $\sum_{I \in \cal G^J} |I| = |J|$ and therefore the subadditivity of the function $\vf(t)= t^{\alpha}$  implies $\sum_{I \in \cal G^J} |I|^{\alpha} \geq |J|^{\alpha}$. Consequently
$$ G_{\alpha} (K_1)  = \sum_{J \in \cal G_{K_1}} |J|^{\alpha} \leq
\sum_{J \in \cal G_{K_1}} \sum_{I \in \cal G^J} |I|^{\alpha} \leq  G_{\alpha} (K_2).$$
\end{proof}

\begin{lemma}\label{aditgap}
Let  $A_1,\dots,A_k$ be compact subsets of an interval $[c,d]$ and $0< \alpha<1$. Then
$$  G_{\alpha} (A_1\cup\dots\cup A_k) \leq G_{\alpha}(A_1)+\dots+ G_{\alpha}(A_k) + (k-1) (d-c)^{\alpha}.$$
\end{lemma}
\begin{proof}
Using induction, we see that it is sufficient to prove the case $k=2$.
 We can suppose that $\max(A_1 \cup A_2) \in A_1$. Set $A_2^*:= A_2 \cup \{d\}$.
Consider the decomposition  $\cal G_{A_1\cup A_2} = \cal G_{12} \cup \cal G_1 \cup \cal G_2$,
 where  
\begin{align*}
\cal G_{12}&:= \{(u,v) \in \cal G_{A_1\cup A_2}:\ (u,v) \in \cal G_{A_1}\cup \cal G_{A_2}\},\\
\cal G_1&:= \{(u,v) \in \cal G_{A_1\cup A_2}:\  u\in A_1\setminus A_2,\ v \in A_2\setminus A_1\},\\
\cal G_2&:= \{(u,v) \in \cal G_{A_1\cup A_2}:\  u\in A_2\setminus A_1,\ v \in A_1\setminus A_2\}.
\end{align*}
To each $(u,v) \in \cal G_{A_1\cup A_2}$ assign an interval $\vf((u,v))\in \cal G_{A_1}\cup 
\cal G_{A_2^*}$ containing $(u, v)$ in the following way:
$$\vf((u,v)):=\begin{cases}
(u,v)  &\text{if } (u,v) \in \cal G_{12},\\
(u, \min(A_1 \cap (v,b])) &\text{if } (u,v) \in \cal G_1, \\
(u, \min(A_2^* \cap (v,b])) &\text{if } (u,v) \in \cal G_2. 
\end{cases} $$
Since the mapping $\vf$ is clearly injective, we obtain
$$  G_{\alpha}(A_1 \cup A_2) \leq G_{\alpha}(A_1) + G_{\alpha}(A_2^*)  \leq G_{\alpha}(A_1) + G_{\alpha}(A_2) + (d-c)^{\alpha}.$$
\end{proof}

Lemmas~\ref{mongap} and \ref{aditgap} give:

\begin{corollary}\label{adbt}
Any compact subset of a finite union of $\BT$ sets is a $\BT$ set.
\end{corollary}

As in \cite{BN}, we will need an estimate of $G_{\alpha}(f(A))$, where $f(A)$ is compact and
 $f: A \to \R$ is a function with special properties. We could use \cite{BN}*{Theorem 2.1},
 but   Lemma \ref{obra} below which is  much easier is sufficient for our purposes.

\begin{lemma}\label{indu}
Let  $\alpha>0$, $F= \{a_1<a_2<\dots< a_n\} \subset \R$ and let $f: F \to \R$ be injective.
 Then  
\begin{equation}\label{obrkon}
G_{\alpha}(f(F)) \leq \sum_{i=1}^{n-1} |f(a_{i+1}) - f(a_i)|^{\alpha}.
\end{equation}
\end{lemma}
\begin{proof}
We will proceed by induction on $n$. For $n=1,2$ the statement is trivial. So suppose that
 $n>2$ and the lemma ``holds for $n=n-1$''. 

Let $f(F)= \{b_1<\dots< b_n\}$ and $f(a_n)= b_k$. Denoting $F^*:= \{a_1,\dots,a_{n-1}\}$, we have
 in the case $1<k<n$
\begin{multline*}
 G_{\alpha}(f(F)) = G_{\alpha}(f(F^*)) - (b_{k+1}- b_{k-1})^{\alpha} + (b_{k+1}- b_{k})^{\alpha}
 +  (b_{k}- b_{k-1})^{\alpha}\\ \leq  G_{\alpha}(f(F^*)) + \min((b_{k+1}- b_{k})^{\alpha}, (b_{k}- b_{k-1})^{\alpha}) \leq G_{\alpha}(f(F^*)) + |f(a_n) - f(a_{n-1})|^{\alpha}.
\end{multline*}
In the (easier) cases $k=n$ and $k=1$ we obtain the same inequalities:
$$
G_{\alpha}(f(F)) = G_{\alpha}(f(F^*)) + (b_n - b_{n-1})^{\alpha} \leq G_{\alpha}(f(F^*)) + |f(a_n) - f(a_{n-1})|^{\alpha},
$$
$$
G_{\alpha}(f(F)) = G_{\alpha}(f(F^*)) + (b_2 - b_1)^{\alpha} \leq G_{\alpha}(f(F^*)) + |f(a_n) - f(a_{n-1})|^{\alpha}.
$$
By the induction hypothesis, $G_{\alpha}(f(F^*)) \leq \sum_{i=1}^{n-2} |f(a_{i+1}) - f(a_i)|^{\alpha}$,
 and so \eqref{obrkon} follows.
\end{proof}

\begin{lemma}\label{obra}
Let  $A\subset \R$, $f:A\to \R$, $\alpha>0$ and $B>0$. Suppose that $f(A)$ is compact
and
\begin{equation}\label{bezgap}
\sum_{i=1}^{n-1} |f(a_{i+1}) - f(a_i)|^{\alpha} \leq B,\ \ \text{if}\ \ a_1< \dots<a_n,\ a_i\in A,\  i=1,\dots,n.
\end{equation}
Then   $G_{\alpha}(f(A))  \leq B$.
 \end{lemma}
\begin{proof}
Let  $(c_i,d_i)$, $i=1,\dots,p$, be (pairwise different) gaps of $f(A)$. Choose  $F= \{a_1 < a_2<\dots< a_n\}\subset A$ such that $f$ is injective on $F$ and $f(F)=  \bigcup_{i=1}^p \{c_i,d_i\}$. Then Lemma \ref{indu} and \eqref{bezgap} imply
$$  \sum_{i=1}^p  (d_i-c_i)^{\alpha}  \leq G_{\alpha}(f(F)) \leq \sum_{i=1}^{n-1} |f(a_{i+1}) - f(a_i)|^{\alpha} \leq B.$$
 Consequently   $G_{\alpha}(f(A))  \leq B$.
\end{proof}
We will use Lemma \ref{obra} via the following lemma.
\begin{lemma}\label{obra2}
Let $M\geq 0$, $d>0$, let $S \subset \R^2$  be an $M$-Lipschitz graph, and let $\emptyset \neq P \subset S$
 be a compact set with  $\diam P \leq d$. Let $C>0$ and  $\kappa: P \to \R$ satisfy
\begin{equation}\label{abfe}
|\kappa(p_1)-\kappa(p_2)| \leq C \|p_1-p_2\|^2,\ \ \ p_1, p_2 \in P.
\end{equation}
Then  $\kappa(P)$ is a $\BT$ set  and
$$  G_{1/2}(\kappa(P)) \leq  C^{1/2} (M+1) d.$$
\end{lemma}
\begin{proof}
First observe that \eqref{abfe} implies that $\kappa$ is continuous and, hence, $\kappa(P)$ is compact.
We can (and will) suppose that  $S= \{(x,g(x)):\ x \in \R\}$, where  $g: \R \to \R$ is
 an $M$-Lipschitz  function. Set  $\psi(x):= (x, g(x))$ and observe that $\psi$ is an $(M+1)$-Lipschitz mapping. 

Now we will apply Lemma \ref{obra} with $\alpha:= 1/2$, $A:= \psi^{-1}(P)$, $f:= \kappa \circ  \psi$
 and $B:= C^{1/2} (M+1) d$. To this end,  
  consider arbitrary
 elements $a_1<\dots <a_n$ of $A$. For each $1\leq i \leq n-1$ we have
$$ |f(a_{i+1}) - f(a_i)|^{1/2} =  |\kappa (\psi(a_{i+1})) - \kappa (\psi(a_i))|^{1/2}.$$
Since  $\psi(a_{i+1}), \psi(a_i)  \in P$, we obtain by  \eqref{abfe}
$$  |f(a_{i+1}) - f(a_i)|^{1/2} \leq ( C \|\psi(a_{i+1} )-\psi(a_i)\|^2   )^{1/2} \leq  
C^{1/2} (M+1) (a_{i+1}-a_i).$$
As  $\sum_{i=1}^{n-1} (a_{i+1}-a_i)\leq \diam(A)\leq  d$, we have
$$\sum_{i=1}^{n-1} |f(a_{i+1}) - f(a_i)|^{1/2} \leq C^{1/2} (M+1) d  =B.$$
So Lemma \ref{obra} implies 
 $$G_{1/2}(\kappa(P))= G_{1/2}(f(A)) \leq B=  C^{1/2} (M+1) d. $$

Since $\cal H^{1}(P)< \infty$, \eqref{abfe} and \cite{Fa}*{Proposition 2.2} imply that $\cal H^{1/2}(\kappa(P))<\infty$,
 and consequently $\kappa(P)$ is a $\BT$ set.
 \end{proof}

\subsection{Critical points of Lipschitz and DC functions}

  Let $f$ be a locally Lipschitz function on an open $G\subset \R^d$. Then 
    $$f^0(a,v) : = \limsup_{z \to a, t \to 0+}\ \frac{f(z+tv)-f(z)}{t}$$
   is the {\it Clarke derivative} of $f$ at $a\in G$ in the direction $v\in \R^d$ and
   $$ \partial f(a) := \{x^* \in (\R^d)^*:\ x^*(v) \leq  f^0(a,v)\ \ \text{for all}\ \ v \in \R^d\}$$
   is the {\it Clarke subdifferential} of $f$ at $a$.
	Since we identify $(\R^d)^*$ with $\R^d$ in the
standard way, we sometimes consider $\partial f(a)$ as a subset of $\R^d$.
	
	If $G$ is a convex set, $f$ is a continuous convex function on $G$ and $a \in G$,
	then (see \cite{Cl}*{Proposition~2.2.7})  
	\begin{equation}\label{rode}
	 f'_+(a,v) =  f^0(a,v) \ \ \text{for each}\ \  v \in \R^d
	\end{equation}
 and	the Clarke subdifferential   $\partial f(a)$
	 coincides with the classical subdifferential from convex analysis.
	We will need also the easy fact that, under the above conditions, 
	\begin{equation}\label{veldiam}
	\diam \partial f(a) \geq  f'_+(x,v) + f'_+(x,-v) \  \text{for each}\  \
	 v \in S_{\R^d},
	\end{equation}
	which follows easily e.g.\ from \cite{Cl}*{Proposition~2.1.2~(b)} and \eqref{rode}.
	 
   We shall use the following standard terminology (see e.g. \cite{Fu}).
   \begin{definition}\label{critreg}
   Let $f$ be a locally Lipschitz function on an open set $\emptyset \neq G\subset \R^d$.  Then we say 
    that $a\in G$ is a {\it regular point} of $f$ if $0 \notin \partial f(a)$. If $0 \in \partial f(a)$,
   we say that $a$ is a {\it critical point} of $f$. The set of all critical points of $f$ will be denoted by 
    $\Crit(f)$. By the set of {\it critical values} of $f$ we mean the set $\cv(f) := f(\Crit(f))$.
   \end{definition}
    
 We will need the following easy lemma (see \cite{RZ1}*{Lemma~2.5}).   
\begin{lemma}\label{charcrit}
Let $f$ be a locally Lipschitz function on an open set $G \subset \R^d$ and $a \in G$. Then the following conditions are equivalent.
\begin{enumerate}
\item $a \notin \Crit(f).$
\item
 There exist $\delta >0$, $\varepsilon >0$, and $v \in \R^d$ such that
$$  \frac{f(x+tv) - f(x)}{t} < - \varepsilon, \ \ \ \text{whenever}\ \ \ t>0, x \in B(a,\delta), x+tv \in B(a,\delta).$$
\end{enumerate}
\end{lemma}

Lemma \ref{charcrit} immediately implies the well-known fact that
   \begin{equation}\label{crcl}
   \Crit(f) \ \ \text{is closed}\ \ \text{in}\ \ G.
   \end{equation}

The following	lemma is a refined version of \cite{RZ1}*{Lemma 3.1}. 

\begin{lemma}\label{critdc2}
Let $f$, $g$ be convex functions on an open convex  set $C \subset \R^d$, and let $d:= f-g$.
 Assume that $\ep>0$, $x\in C$, $v\in \R^d$ are such that $d'_+(x,v) < -5\ep$,
\begin{equation}\label{mzl}
 f'_+(x,v) + f'_+(x,-v) \leq \ep\ \ \text{and}\  \ g'_+(x,v) + g'_+(x,-v) \leq  \ep.
\end{equation}
Then  $x \notin \Crit(d)$.
\end{lemma} 
\begin{proof}
Since $f$ is convex, we have by \eqref{rode}
 $$ f'_+(x,v) =  f^0(x,v) = \limsup_{y \to x, t \to 0+} \frac{f(y+tv) - f(y)}{t}$$
and
 $$  -f_+'(x,-v) =-f^0(x,-v)= -\limsup_{y \to x, t \to 0+} \frac{f(y-tv) - f(y)}{t} $$
 $$= \liminf_{y \to x, t \to 0+} \frac{ f(y)- f(y-tv)}{t}= \liminf_{z \to x, t \to 0+}\frac{f(z+tv) - f(z)}{t}.$$
 Consequently  there exists $\delta_1>0$ such that
 \begin{equation}\label{shzd}
 -f_+'(x,-v) - \ep < \frac{f(y+tv) - f(y)}{t} < f_+'(x,v) + \ep,
\end{equation}
 whenever $t>0$, $y \in B(x,\delta_1)$ and  $y+tv \in B(x,\delta_1)$.
Since \eqref{mzl} implies $f_+'(x,v)- 2 \ep \leq  -f_+'(x,-v) - \ep $, 
 \eqref{shzd} implies
 $ |(f(y+tv) - f(y))/t  -  f'_+(x,v)| \leq 2\ep.$ Proceeding by the same way with
 $g$ instead of $f$, we obtain that there exists $\delta>0$ such that, for each $t>0$ and
 $y$ with $y \in B(x,\delta)$ and  $y+tv \in B(x,\delta)$, we have
$$ \left|\frac{f(y+tv) - f(y)}{t}  -  f'_+(x,v)\right| \leq  2 \ep\ \ \text{and}\ \ 
 \left|\frac{g(y+tv) - g(y)}{t}  -  g'_+(x,v)\right| \leq 2\ep,$$
 and consequently
$  |(d(y+tv) - d(y))/t  -  d'_+(x,v)| \leq  4 \ep$, which implies
 $$ \frac{d(y+tv) - d(y)}{t} < -5 \ep + 4 \ep = -\ep.$$
Thus we obtain  $x \notin \Crit(d)$ by Lemma \ref{charcrit}.
\end{proof}

We will also need the following version of \cite{RZ3}*{Corollary~4.5}.

\begin{lemma}\label{poklip}
Let $f$ be a convex function on an open convex set $C \subset \R^2$. Suppose that $f$ is $L$-Lipschitz
 ($L>0$) and let $0< \ep <1$ be given. Then the set
$$ A:= \{x\in C:\ \diam (\partial f(x)) > \ep\}$$
 can be covered by $m$ graphs of $M$-Lipschitz functions, where
$$  m \leq \frac{16\, L}{\ep}\ \ \text{and}\ \  M= \frac{8\, L}{\ep}.$$
\end{lemma}
\begin{proof}
Let $\pi_1$ and $\pi_2$ be the coordinate projections in $\R^2$.
Clearly  $A \subset A^1 \cup A^2$, where
$$ A^1= \{x\in C:\ \diam( \pi_1 \partial f(x)) > \ep/2\},\  
 A^2= \{x\in C:\ \diam( \pi_2 \partial f(x)) > \ep/2\}.$$
Set
$$ n:= \lfloor 8 L/\ep\rfloor \ \ \ \text{and}\ \ \ t_k:= -L + k \ep/4,\ \ k=0,\dots, n+1,$$
and, for $1\leq k \leq n$, 
$$ A^2_k:= \{x \in A_2:\  \pi_2 \partial f(x) \cap (-\infty, t_k)\neq \emptyset,\ \
\pi_2 \partial f(x) \cap (t_{k+1},\infty)\neq \emptyset\}.$$
 Since, for each $x \in C$, $\pi_2 \partial f(x) \subset [-L,L] \subset [t_0,t_{n+1}]$,
it is easy to see that  $A^2 = \bigcup_{k=1}^{n} A^2_k$.
 We will show that 
\begin{equation}\label{adk}
\text{each $A^2_k$ is a subset of a graph of an $(8L/\ep)$-Lipschitz function.}
\end{equation}
So, fix $ 1\leq k \leq n$ and $u= (u_1,u_2) \in A^2_k$, $v= (v_1,v_2) \in A^2_k$.
 Without any loss of generality we can suppose that $v_2 \geq u_2$.
By the definition of $A^2_k$, we can choose $h= (h_1, h_2) \in \partial f(u)$
 and $d= (d_1,d_2) \in \partial f(v)$  such that $d_2 < t_k$ and $h_2>t_{k+1}$.
 Since $\partial f$ is a monotone operator (see e.g. \cite{Cl}*{Proposition~2.2.9}), we have
$$ 0 \leq \langle v-u, d-h \rangle = (v_1-u_1) (d_1-h_1) + (v_2-u_2) (d_2-h_2).$$
As $h_2- d_2 > \ep/4$, $v_2-u_2 \geq 0$ and $|d_1 -h_1| \leq 2L$, we obtain 
$ |v_2-u_2| \leq (8L/\ep) |v_1-u_1|$. Since any $C$-Lipschitz function from a subset of $\R$
 can be extended to a  $C$-Lipschitz function on $\R$, \eqref{adk} follows.
Consequently $A^2$ can be covered by $n$ graphs of $M$-Lipschitz functions. Proceeding quite analogically, we obtain that the same holds for $A_1$ and our assertion follows.
\end{proof}
Using \eqref{veldiam}, we easily obtain the following corollary.
\begin{lemma}\label{kvant}
Let $\Omega \subset \R^2$ be an open convex set,  $f$ a Lipschitz convex function on $\Omega$,  and $\ep>0$. Then  the set 
\begin{eqnarray*}
A:= \{x \in \Omega: \ f'_+(x,v)+ f'_+(x,-v) > \ep \text{ for some } v \in S_{\R^2}\} 
\end{eqnarray*}
 can be covered by finitely many Lipschitz graphs.
\end{lemma}

\begin{lemma}\label{pokr}
Let  $d$ be a locally DC function on an open  $G \subset \R^2$ and $\alpha>0$. Assume that 
  for each $x \in G$ there exists $v \in S_{\R^2}$ such that
 $d'_+(x,v) < - \alpha$. Then for each $z \in G$ there exists $r>0$ such that
$\Crit(d) \cap B(z,r)$ can be covered by finitely many  Lipschitz graphs.
\end{lemma}
\begin{proof}
Let $z\in G$ be arbitrary. Since each convex function on an open convex set is locally
 Lipschitz, we can choose $r>0$ and convex Lipschitz functions  $f$, $g$ on
 $B(z,r)$ such that  $d(x)= f(x)-g(x),\ x \in B(z,r)$. Set $\ep:= \alpha/5$ and apply
 Lemma \ref{kvant} to $f$ and to $g$. We obtain Lipschitz graphs $P_1,\dots, P_s$
 such that, for each $x\in B(z,r) \setminus (P_1\cup\dots\cup P_s)$ and
 $v \in S_{\R^2}$, we have
\begin{equation}\label{fag}
 f'_+(x,v)+ f'_+(x,-v) \leq \ep\ \ \ \text{and}\ \ \ g'_+(x,v)+ g'_+(x,-v) \leq \ep.
\end{equation}
Now consider an arbitrary $x \in B(z,r) \setminus (P_1\cup\dots\cup P_s)$ and
 choose $v \in S_{\R^2}$ with  $d'_+(x,v) < -\alpha = -5 \ep$. Then \eqref{fag}
 and Lemma \ref{critdc2} imply that $x \notin \Crit(d)$ and therefore $\Crit(d) \cap B(z,r) \subset P_1\cup\dots\cup P_s$.
\end{proof}

\subsection{Some facts about distance functions}

Here we recall some well-known facts about distance functions, their critical points,
 and critical values.

Recall that if $X$ is a metric space and $\emptyset \neq F \subset X$ a closed set then 
$d_F = \dist(\cdot,F)$ and  $S_r= \{x \in X:\ d_F(x)=r\}$, $r>0$.  It is well-known that $d_F$ is $1$-Lipschitz.
We will use the following obvious observation.
\begin{lemma}\label{reduk}
Let $a\in X$, $r>0$ and $B(a,r) \cap F \neq \emptyset$. Then
$$  d_F(x)= d_{F \cap \overline{B}(a,3r)}(x)\ \ \ \text{for each}\ \ \ x \in B(a,r).$$
\end{lemma}

In the rest of this subsection, let $X=\R^d$ and let $\emptyset\neq F\subset\R^d$ be closed. 
In this case,  $d_F$ is locally semiconcave on $\R^d \setminus F$; more precisely
 (see e.g.\ the proof of  \cite{BZ}*{Theorem~5.3.2}) the following holds.
\begin{lemma}\label{semdist}
 If $\emptyset \neq F \subset \R^d$ and $d_F(x)=: \delta >0$ then
  there exists a convex function $\gamma$ on $B(a, \delta/2)$ such that
 $ d_F(x) = \frac{2}{\delta}\|x\|^2 - \gamma(x)$  for each $x \in B(a, \delta/2)$.
\end{lemma}
 Further, it is easy to see  that for each $x \in \R^d \setminus F$
\begin{equation}\label{dmj}
(d_F)'_+(x,v) = -1\ \ \text{for some }\ \ v \in S_{\R^d}.
\end{equation}
Lemma \ref{charcrit}  easily yields that $F \subset \Crit(d_F)$ and consequently $0\in \cv(d_F)$.

Recall that the sets  $T_F$ and $L_F$ are defined in the Introduction, namely $T_F$ ($L_F$) is  the set of all
 $r>0$ for which $S_r$ is nonempty and it is not a topological (Lipschitz, resp.) $(d-1)$-dimensional manifold. 
Further recall that if $r\in (0,\infty)\setminus \cv(d_F)$, then Clarke's implicit theorem
 for Lipschitz functions implies (cf.\ \cite{Fu}*{Theorem~3.1}) that $S_r$ is either empty
 or a $(d-1)$-dimensional Lipschitz manifold.  Consequently,
\begin{equation}\label{tlcv}
 T_F \subset L_F \subset \cv(d_F).
\end{equation}
Further, if  $x \in \Crit(d_F)$, then  $d_F(x) \leq \diam (F)$ (see \cite{Fu}*{p.~1038} for a more precise estimate), and so
\begin{equation}\label{cvdiam}
\cv(d_F) \subset [0, \diam(F)].
\end{equation}

Finally, we will essentially use the following immediate consequence of Ferry's inequality
 (\cite{Fe}*{Proposition~1.5}), see \cite{Fu}*{Lemma~4.3}:
\begin{equation}\label{ferryn}
|d_F(v) - d_F(w)| \leq (2 \min(d_F(v), d_F(w)))^{-1} \|v - w \|^2, \ \ \text{if}\ \ v,w \in \Crit(d_F) \setminus F.
\end{equation}

\section{Results in Euclidean spaces}
 
The following three lemmas will be used to prove both Proposition \ref{mkrb} on smallness of $\Crit(d_F)$
 and subsequent results on smallness of $\cv(d_F)$.

\begin{lemma}\label{nakoulib}
Let $\emptyset \neq F\subset \R^2$ be a compact set, let $a\in \R^2$ and $d_F(a)\geq \delta>0$. Then
 the set   
$$  T:= \Crit (d_F) \cap \overline B(a, \delta/3)$$
 can  be covered by 
 $m$  $144$-Lipschitz graphs $G_1, \dots, G_m$,
 with $m \leq 288$.
\smallskip

Consequently,  $\cal H^1(T) \leq 10^5\cdot \delta$.
\end{lemma}
\begin{proof}
By Lemma \ref{semdist} there exists a convex function $\gamma$ on $B(a, \delta/2)$ such that
 $ d_F(x) = \frac{2}{\delta}\|x\|^2 - \gamma(x)$  for each $x \in B(a, \delta/2)$, 
  which implies
 that the function
$$  g(x)= \frac{2}{\delta}\|x-a\|^2 - d_F(x),\ \ x \in B(a, \delta/2),$$
 is convex.  Since  $d_F$ is $1$-Lipschitz, it is easy to see that $g$ is $3$-Lipschitz on $B(a, \delta/2)$.
Set
$$ Z:= \{x\in B(a, \delta/2) :  \diam(\partial g(x)) > 1/6\}.$$
We have  $d_F(x)= \frac{2}{\delta}\|x-a\|^2 - g(x),\ x  \in B(a, \delta/2)$.
 Now consider an arbitrary $x \in B(a, \delta/2) \setminus Z$.
  Since $x \notin F$, by \eqref{dmj}  there exists a unit $v\in \R^2$
 such that $(d_F)'_+(x,v)= -1$. As $x \notin Z$, we have $g'_+(x,v) + g'_+(x,-v) \leq 1/6$ (see 
\eqref{veldiam}).
 So, since the convex function $\vf:= \frac{2}{\delta}\|\cdot-a\|^2$ is differentiable at $x$,
 we can apply Lemma \ref{critdc2} (with $\ep= 1/6$) and obtain that $x \notin \Crit(d_F)$.
  Consequently $T \subset Z$.
  Lemma \ref{poklip}
 implies that  $Z$ (and so also $T$) can   be covered by 
 $m$  $144$-Lipschitz graphs $G_1, \dots, G_m$,
 where $m \leq 288$.

To estimate $\cal H^1(T)$, observe that each set $G_n \cap T$ is the image of a set $D_n \subset \R$, which is 
 contained either in $[a_1-\delta/3, a_1+ \delta/3]$ or in $[a_2-\delta/3, a_2+ \delta/3]$
 under a   $145$-Lipschitz  mapping $\vf: D_n \to \R^2$, and consequently
 $\cal H^1(G_n \cap T) \leq 145\cdot \delta$. Therefore, $\cal H^1(T) \leq 288\cdot 145\cdot \delta\leq 10^5 \cdot \delta$.
\end{proof}

\begin{remark}\label{noef}
We make nowhere any effort to find optimal multiplicative constants, which are surely much smaller.
\end{remark}

The following lemma  will be applied with $\alpha=1$ and  $\alpha=3/2$.

\begin{lemma}\label{sumint}
Let $\alpha \geq 1$,  $D>0$ and $g$ be a nonincreasing function on $(0,D]$ with $g(D)=0$. Set $\delta_k:= D 2^{-k}$, $k=0,1,2\dots$.
 Then the following conditions are equivalent.
 \begin{enumerate}
 \smallskip
 
 \item[(i)]   $\sum_{k=0}^{\infty}  g(\delta_k) \cdot (\delta_k)^{\alpha} < \infty.$
 \smallskip
 
 \item[(ii)]   $\sum_{k=0}^{\infty}  (g(\delta_{k+1})- g(\delta_k)) \cdot (\delta_k)^{\alpha} < \infty.$
 \smallskip
 
 \item[(iii)]  $\int_0^D g(x) x^{\alpha-1}\ dx \ < \infty.$
 \end{enumerate}
\end{lemma}
\begin{proof}
Set  $u_k:= g(\delta_{k+1})- g(\delta_k)$, $k=0,1,\dots$, and $r:= 2^{- \alpha}$. Then $s_k:= \sum_{i=0}^k u_i= g(\delta_{k+1})$
 and $(\delta_k)^{\alpha} = D^{\alpha} r^k$,  $k=0,1,\dots$. It is well-known (see, e.g. 
\cite{WZ}*{p.~231}) that (for any sequence $(u_k)$
  and $0<r<1$)
  \begin{equation}\label{abel}
   \sum_{n=0}^{\infty} u_nr^n = (1-r) \sum_{n=0}^{\infty} s_n r^n,
  \end{equation}
   whenever one of the series converges. Consequently
   \begin{equation}\label{sdifer}
   \sum_{k=0}^{\infty}  (g(\delta_{k+1})-g(\delta_k) ) 2^{-k\alpha} < \infty
   \end{equation}
   if and only if

   \begin{equation}\label{bezdif}
   \sum_{k=0}^{\infty}  g(\delta_{k+1}) 2^{-k \alpha} < \infty.
   \end{equation}
   Since  (i) is clearly equivalent to \eqref{bezdif} and (ii) to \eqref{sdifer}, we proved that (i)$\iff$(ii).
   
In order to prove that (i)$\iff$(iii), denote $I_k:= \int_{\delta_{k+1}}^{\delta_k} g(x) x^{\alpha-1}   \, dx $
    and observe that   $\int_0^D g(x) x^{\alpha-1}\, dx = \sum_{k=0}^{\infty} I_k$. By properties of $g$, we have
    $$ g(\delta_k) (\delta_{k+1})^{\alpha-1} \leq g(x)  x^{\alpha-1} \leq g(\delta_{k+1}) (\delta_k)^{\alpha-1},\quad x \in [\delta_{k+1}, \delta_k],
  $$ 
  and consequently
\begin{align*}
g(\delta_k) (\delta_k)^{\alpha} 2^{-\alpha} = g(\delta_k) (\delta_{k+1})^{\alpha-1} \delta_{k+1} \leq I_k
 \leq& g(\delta_{k+1}) (\delta_{k})^{\alpha-1} \delta_{k+1}\\ 
=& 2^{\alpha-1} g(\delta_{k+1}) (\delta_{k+1})^{\alpha},
\end{align*}
which easily implies that (i)$\iff$(iii).
\end{proof}

\begin{lemma}\label{peter}
Let  $\emptyset \neq F \subset \R^2$ be compact. Set $D:= \diam F$,  $\delta_n:= D\,  2^{-n},\ n=0,1,\dots$.  For  $0<\alpha < \beta$, denote $H(\alpha,\beta):= \{x\in \R^2:\ \alpha \leq d_F(x)\leq \beta\}$ and  set $H_n:= H(\delta_{n+1}, \delta_n)$.

Then, for each $n$, there exists a finite set $P_n \subset H_n$ such that
  \begin{equation}\label{covhn}
\text{the system}\ \   \tilde{\cal B}_n:= \{ \overline{B}(x, \delta_n/ 8):\ x \in P_n\}\ \ \ 
\text{covers}\ \ \  H_n
\end{equation}
 and, for  $p_n:= \card P_n$, we have
\begin{equation}\label{sumaan}
\sum_{n=0}^{\infty} p_n (\delta_n)^2   \leq  10^{5} \cdot D^2.
\end{equation}
\end{lemma}
\begin{proof}
Set  $\tilde H_n:= H(\delta_{n+1}- \delta_n/8, \delta_n+ \delta_n/8),
 \ \ n=0,1,2\dots,$ 
and observe that both systems
 $\tilde H_0, \tilde H_2, \tilde H_4,\dots$ and     $\tilde H_1, \tilde H_3, \tilde H_5,\dots$
 are disjoint, $\tilde H_n \subset B(x_0, 3D)$ for each  $n$ and $x_0 \in F$, and consequently
\begin{equation}\label{sumamer}
\sum_{n=0}^{\infty} \lambda_2(\tilde H_n) \leq 2 \pi 9 D^2.
\end{equation}
Now let  $n\geq 0$ be fixed. Since the system of closed balls
$$\cal B_n:= \{ \overline{B}(x, \delta_n/ 40):\ x \in H_n\}$$
covers $H_n$, by the $5r$-covering theorem (see, e.g., \cite{Ma}*{Theorem~2.1}) there exists a (necessarily finite) set $P_n \subset H_n$
 such that 
$$\text{the system}\ \  \cal B_n^*:= \{ \overline{B}(x, \delta_n/ 40):\ x \in P_n\}\ \ \ \text{is disjoint }$$
 and \eqref{covhn} holds.

 Obviously  $\bigcup \cal B_n^* \subset \tilde H_n$ and so
$$ \lambda_2(\bigcup \cal B_n^*) = p_n \cdot \frac{\pi (\delta_n)^2}{40^2} \leq \lambda_2(\tilde H_n).$$
Using also \eqref{sumamer}, we obtain
$$ \sum_{n=0}^{\infty} p_n (\delta_n)^2 \leq \frac{40^2}{\pi} \cdot 18 \pi D^2$$
 which implies \eqref{sumaan}. 
\end{proof}

\begin{proposition}\label{mkrb}
For a compact
 $\emptyset \neq F \subset \R^2$ and $r>0$, set
$$  h(r):=\cal H^1 \{x \in \Crit(d_F): d_F(x) >r\}.$$
 Then 
$$ \int_0^{\infty} h(r)\ dr < \infty.$$
\end{proposition}
\begin{proof}
First note that $h$ is nonnegative and nonincreasing. Let $\delta_n$, $H_n$ and $P_n \subset H_n$
 be as in Lemma \ref{peter}. Then \eqref{covhn} and \eqref{sumaan} hold. For each $n\geq 0$ and $x\in P_n$ set  $T_{n,x}:= \Crit(d_F) \cap \overline B(x,\delta_n/8)$. Since, for $x \in P_n$, we have
 $d_F(x) \geq \delta_{n+1} = \delta_n/2 \geq 3\cdot (\delta_n/8)$, Lemma \ref{nakoulib} implies
 that $\cal H^1(T_{n,x}) \leq 10^5\cdot 3 \cdot ( \delta_n/8)$. Therefore we obtain by \eqref{covhn}
$$  \kappa_n:= \cal H^1(\Crit(d_F) \cap H_n) \leq 10^5 p_n \delta_n$$
 and consequently
$$  h(\delta_{n+1}) - h(\delta_n) \leq \kappa_n \leq 10^5 p_n \delta_n.$$
 
 Thus condition (ii) of Lemma \ref{sumint} with $\alpha=1$ holds by \eqref{sumaan}.
 Consequently condition (iii) of Lemma \ref{sumint} holds and so 
 $ \int_0^{\infty} h(r)\ dr < \infty$ follows since $h(r)=0$ for $r \geq D$ by \eqref{cvdiam}.
\end{proof}

\begin{lemma}\label{nakouli}
Let $\emptyset \neq F\subset \R^2$ be a compact set, let $a\in \R^2$ and $d_F(a)= :\delta>0$. Then
 $d_F(\Crit(d_F) \cap \overline B(a, \delta/3))$ is a $\BT$ set and 
$$ G_{1/2} (d_F(\Crit(d_F) \cap \overline B(a, \delta/3)))  \leq   6 \cdot 10^4 \sqrt{\delta}. $$
\end{lemma}
\begin{proof}
  By Lemma \ref{nakoulib} the set $  T:= \Crit (d_F) \cap \overline B(a, \delta/3)$
 can  be covered by 
 $m$  $144$-Lipschitz graphs $G_1, \dots, G_m$,
 where $m \leq 288$. Thus
	\begin{equation}\label{inklsj}
 d_F(T)\subset \bigcup_{k=1}^m  d_F(T\cap G_k).
\end{equation}
 By \eqref{ferryn} we have
\begin{equation}\label{fefu}
 |d_F(v) - d_F(w)| \leq \delta^{-1} \|v - w \|^2\ \ \text{for every}\ \ v,w \in T.
\end{equation}
Let $1 \leq k \leq m$ be given.
Applying  Lemma \ref{obra2} to $S:= G_k$, $P:= T \cap G_k$, $d:= \delta$, $\kappa:= d_F\restriction_P$ and
 $C:= \delta^{-1}$, we obtain that  $d_F(T\cap G_k)$ is a $\BT$ set with
 $$   G_{1/2}( d_F(T\cap G_k)) =  G_{1/2}(\kappa(P)) \leq \delta^{-1/2}\cdot 145 \cdot \delta = 145 \cdot \sqrt{\delta}.$$
 Observing that $d_F(T) \subset [2\delta/3, 4\delta/3]$ and using \eqref{inklsj}  and Lemma \ref{aditgap}, we obtain that $d_F(T)$ is a $\BT$ set with
$$G_{1/2}(d_F(T)) \leq   m \cdot 145 \sqrt{\delta} + (m-1) \sqrt{\delta} \leq 288\cdot 146\cdot \sqrt{\delta}
 \leq 6\cdot 10^4\cdot \sqrt{\delta}.$$
\end{proof}

As an easy corollary, we obtain the following part of Theorem \ref{odrn}
\begin{proposition}\label{zlfud}
If  $\emptyset \neq F \subset \R^2$ is compact and $\ep>0$, then  $\cv(d_F) \cap [\ep, \infty)$ is a
 $\BT$ set. 
In particular, each of the  sets  $\cv(d_F) \cap [\ep, \infty)$,  $L_F \cap [\ep, \infty)$
and  $T_F \cap [\ep, \infty)$ has zero $1/2$-dimensional  Minkowski content.

Consequently, $\cv(d_F)$ is a countable union of sets of zero $\frac 12$-dimensional Min\-kow\-ski content and so $\dim_P(\cv(d_F))\leq \frac 12$.

\end{proposition}
\begin{proof}
For each $z \in \R^2 \setminus F$, we choose by using Lemma~\ref{nakouli} an $r(z)>0$ such that
 $d_F(\Crit(d_F) \cap \overline B (z,r(z)))$ is a $\BT$ set. The set $C:= (d_F)^{-1}([\ep,\diam F])$ is clearly  compact. Consequently we can choose points
 $z_1,\dots, z_p \in C$ such that  $C \subset  \bigcup_{i=1}^p  \overline B (z_i,r(z_i))$.
 Therefore
\begin{equation}\label{konpokr}
 d_F(\Crit(d_f) \cap C) \subset \bigcup_{i=1}^p  d_F(\Crit(d_F) \cap \overline B (z_i,r(z_i))).
\end{equation}
By  \eqref{cvdiam},  $\cv (d_F) \cap [\ep,\infty)= d_F(\Crit(d_F) \cap C)$. So, 
 $\cv (d_F) \cap [\ep,\infty)$ is compact and  \eqref{konpokr} and  Corollary \ref{adbt} imply
 that $\cv (d_F) \cap [\ep,\infty)$
is a $\BT$ set. The rest of the assertion follows by \eqref{tlcv}, \eqref{nulmin} and \eqref{dimp}.
\end{proof}
\begin{remark}\label{inde} 
 We have proved (see \eqref{nuldim} and \eqref{nedi}), by different arguments,  Fu's \cite{Fu} results saying that
$\cal \cal H^{1/2}(\cv(d_F))=0$ and  that  $\cv(d_F) \cap [\ep,\infty)$ has the entropy (=upper Minkowski) dimension at most $\frac 12$.   
\end{remark}

To prove whole Theorem \ref{odrn}, we need the following proposition, whose idea of the construction  goes back to \cite{Fe}.

\begin{proposition}\label{konstr}
Let $K \neq \emptyset$ be a $\BT$ set with $0<a:= \min K \leq b:= \max K$.
 Set  $g(y):= \sqrt{2b}\cdot  G_{1/2} (K \cap [a,y]),\ \ y \in K,$
 and 
$$ F:= \{(g(y), \pm y): y \in K\}.$$
Then $F$ is a compact set, $K\subset L_F$, and so $K\subset \cv(d_F)$.
\end{proposition}
\begin{proof}
 It is easy to check that $g$ is strictly increasing and continuous
 on $K$ and therefore $F$ is compact. Now consider arbitrary $u,v \in K$, $u<v$.
 Denoting  $\cal G:= \cal G_{K\cap [u,v]}$, we have
\begin{equation}\label{sumy}
 v-u = \sum_{I \in \cal G} |I| \leq  \left(\sum_{I \in \cal G} \sqrt{|I|} \right)^2 = (2b)^{-1}
 (g(v)-g(u))^2.
\end{equation}
  We will show that the metric projection $\pi_F$ to $F$ satisfies
\begin{equation}\label{proj}
 \text{ $\pi_F ((g(v),0)) = \{(g(v),v), (g(v), -v)\}$ for each $v \in K$.}
\end{equation}
To this end, consider $u \in K$, $u \neq v$. If  $u< v$, using  $\eqref{sumy}$ we obtain
$$ v^2-u^2 = (v-u)(v+u) < 2b (v-u)  \leq  (g(v)-g(u))^2 $$
 and consequently
\begin{equation}\label{vzduv}
\|(g(v),0) - (g(u),u)\| = \sqrt{(g(v)-g(u))^2 +u^2}  >v =  \|(g(v),0) - (g(v),v)\|.
\end{equation}
 Since \eqref{vzduv} is obvious for $u>v$, \eqref{proj} easily follows.
 
 Further note that
 \begin{equation}\label{nlnv}
 \text{if $v\in K$ and $z> g(v)$, then $\dist((z,0), F)>v$.}
 \end{equation}
 To this end, consider a point $W\in F$; we can suppose without loss of generality that $W= (g(w),w)$ for some $w \in K$. Using the monotonicity of $g$, it is easy to see that 
  $\|(z,0)- (g(w),w)\| >v$ if $w \geq v$. If $w<v$, observe that clearly
	$\|(z,0)- (g(w),w)\|> \|(g(v),0)- (g(w),w)\|$ and $\|(g(v),0)- (g(w),w)\|>v$ holds by \eqref{vzduv}. Using the compactness of $F$, we obtain \eqref{nlnv}.

In order to prove that $K \subset L_F$, suppose to the contrary that there exists a $v\in K \setminus L_F$.
Now set
$$   t(h):= v - \sqrt {v^2-h^2},\ \ \ 0<h<v, $$
 and observe that $\|(g(v),\pm v) - (g(v)+h, \pm t(h))\|= v$. Using also  \eqref{nlnv}, we obtain that
$$ \dist((g(v)+h, \pm t(h)), F) \leq v\ \ \ \text{and}\ \ \ \dist((g(v)+h, 0)), F) >v,$$
 and consequently there exist (for  $h \in (0,v)$) numbers $t^+(h) \in (0, t(h)]$, $t^-(h)\in [-t(h),0)$ such that
$$    p^{\pm}(h):= (g(v)+h, t^{\pm}(h)) \in S_v(F).$$
Since $(g(v),0) \in S_v(F)$ and $v \notin L_F$, there exists an open neigbourhood $U$
 of $(g(v),0)$, an open interval $I \subset \R$, $K>0$ and a bijection $\vf:I \to U\cap S_v(F)$
  such that both $\vf$ and $\vf^{-1}$ are $K$-Lipschitz. Choose $h_0 >0$
	such that  $p^{\pm}(h)\in U$ for each $0<h<h_0$ and define
	$$    \tau^{\pm}(h) := \vf^{-1}(p^{\pm}(h)),\ \ \  0<h<h_0.$$
	For each   $0<h<h_0$, we know that
	$$  A(h):= \vf(\conv\{\tau^+(h), \tau^-(h)\}) \subset S_v(F)$$  
	 is connected. So, since  $p^{\pm}(h)\in A(h)$, using \eqref{nlnv}
 we obtain that there exists $y(h) \leq g(v)$ such that $(y(h),0) \in A(h)$.
Since clearly  $|\tau^+(h) - \tau^-(h)| \leq  2Kt(h)$, we obtain
$$ h \leq \|(y(h),0)- p^+(h)\|\leq  K  |\tau^+(h) - \tau^-(h)|      \leq 2 K^2 t(h),\ \ \ 0<h<h_0,$$
 which contradicts  the fact that  $\lim_{h\to 0} t(h)/h = 0$.
\end{proof}

\begin{remark}\label{Fmala}
The set $F$ from Proposition \ref{konstr} is clearly contained in the ball 
$\overline B(0,b+ \sqrt{2b}\cdot  G_{1/2} (K))$. 
\end{remark}

\begin{proof}[Proof of Theorem~\ref{odrn}]
 By \eqref{tlcv}, (i)$\implies$(ii). If (ii) holds, then $\overline{A}\subset \cv(d_F)\cap [\ep,\infty)$, since $\cv(d_F)\cap [\ep,\infty)$ is closed. So $\overline{A}$ is a $\BT$ set
 by Proposition~\ref{zlfud} and Lemma~\ref{mongap}.
 The implication  (iii)$\implies$(i) follows by Proposition~\ref{konstr} (applied
 to $K:= \overline{A}$).
\end{proof}

\begin{remark}
Note that it is not difficult to construct a $\BT$ set of Hausdorff dimension $\frac 12$. Hence, Theorem~\ref{odrn} gives another argument for the result of Fu \cite{Fu}*{\S5.2} that the case $\dim_H(\cv(d_F))=\frac 12$ is possible in the planar case. For a stronger result, see Theorem~\ref{TF}. 
\end{remark}

To prove Theorem \ref{TF} concerning sets $T_F$, we will need the following supplement
 to Proposition \ref{konstr}.

\begin{lemma}  \label{konstr_TF}
Let $K,F$ be as in Proposition~\ref{konstr} and let $y\in K$ be such that for any $\ep>0$ there exist $x\in K\cap(y-\ep,y)$ and a gap $I\subset (x,y)$ of $K$ with $|I|\geq \frac{4y}{b}(y-x)$. Then $y\in T_F$.
\end{lemma}

\begin{proof}
Recall that $(g(y),0)\in S_y(F)$. 
Given $n\in\N$, there exist by assumption a point $x_n\in K\cap(y-\frac 1n,y)$ and a gap $I=(u_n,v_n)\subset (x_n,y)$ of $K$ with $v_n-u_n\geq\frac{4y}b(y-x_n)$. We shall show that $d_F((z_n,0))\geq y$, where $z_n:=\frac{g(u_n)+g(v_n)}2$. Using \eqref{vzduv} we find that for any $w\in K$ such that $w\leq u_n$,
$$\|(z_n,0)-(g(w),w)\|^2\geq w^2+(g(u_n)-g(w))^2+(z_n-g(u_n))^2>u_n^2+(z_n-g(u_n))^2,$$
hence, using \eqref{sumy}, we obtain
\begin{eqnarray*}
d_F((z_n,0))^2&\geq&u_n^2+\frac{(g(v_n)-g(u_n))^2}4\geq u_n^2+\frac{2b(v_n-u_n)}{4}\\
&\geq&x_n^2+\frac b2 \frac{4y}{b}(y-x_n)\geq x_n^2+(y+x_n)(y-x_n)=y^2.
\end{eqnarray*}
Thus we have $d_F((z_n,0))\geq y$. On the other hand,
$d_F(g(u_n),s)<y$ and $d_F(g(v_n),s)<y$ whenever $|s|<y$.
Observe also that if $g(y)-y<t<g(y)$ and $h(g(y)-t)<|s|<y$, where 
$$h(q):=y-\sqrt{y^2-q^2},\quad |q|<y,$$
then $d_F((t,s))\leq\min\{\|(t,s)-(g(y),-y)\|,\|(t,s)-(g(y),y)\|\}<y$.
Consequently, for all sufficiently large $n$, there must be a (nonempty) component $C_n$ of $S_y(F)$ contained in 
$$\{(t,s):\, g(u_n)<t<g(v_n),\, -h(g(y)-t)\leq s\leq h(g(y)-t)\},$$
which therefore satisfies $(g(y),0)\not\in C_n$, but 
$\dist((g(y),0),C_n)\to 0$, $n\to\infty$. This excludes the possibility of $S_y(F)$ being a topological $1$-manifold.
\end{proof}

\begin{theorem}  \label{TF}
There exists a compact set $F\subset\R^2$ with $\dim_H(T_F\cap[1,2])=\frac 12$.
\end{theorem} 

\begin{proof}
Given $\alpha\in(0,\frac 12)$, let $C(\alpha)\subset[0,1]$ be the Cantor set obtained by removing iteratively the middle open interval of proportional length $1-2\alpha$ from each component at each step (alternatively, $C(\alpha)$ can be described as the self-similar set determined by the two similarities $x\mapsto\alpha x$, $x\mapsto 1-\alpha+\alpha x$). We have $\dim_H(C(\alpha))=\frac{\ln 2}{-\ln\alpha}$ (see \cite{Ma}*{\S4.10} or \cite{Fa}*{Theorem~9.3}). A direct calculation yields that
$$G_{1/2}(C(\alpha))=\sum_{n=0}^\infty 2^n\sqrt{(1-2\alpha)\alpha^n}=\frac{\sqrt{1-2\alpha}}{1-2\sqrt{\alpha}}.$$
Set $K_n:=1+2^{-n}+2^{-n}C(\frac 14-\frac 1{5n})$, $n=1,2,\dots$, and
$$K:=\{1\}\cup\bigcup_{n=1}^\infty K_n\cup\{16\}.$$
Since $G_{1/2}(K_n)=2^{-n/2}G_{1/2}(C(\frac 14-\frac 1{5n}))$, it follows easily that
$$G_{1/2}(K)=\sum_{n=1}^\infty G_{1/2}(K_n)+\sqrt{14}<\infty.$$
Thus, Proposition~\ref{konstr} can be applied, and let $F\subset\R^2$ be the associated compact set. Each set $K_n$ has the property that for any two points $x,y\in K_n$, $x<y$, there exists a gap $I\subset(x,y)$ of $K_n$ with $I>\frac 12(y-x)$ (this follows from the construction of $C(\alpha)$ where the middle part of proportional length $1-2\alpha$, bigger than $\frac 12$ in our case, is removed at each step).
Thus, for any left accumulation point $y$ of $K$ and any $\ep>0$ there exists a point $x\in K\cap(y-\ep,y)$ and a gap $I\subset(x,y)$ of $K$ with $|I|\geq\frac 12(y-x)\geq\frac{4y}{b}(y-x)$ (note that $y\leq 2$ and $b=16$ in our case). Lemma~\ref{konstr_TF} thus yields that $T_F\cap[1,2]$ contains the set $K^*$ of all left accumulation points of $K$. Finally, note that
$$\dim_H K=\sup_n\dim_HK_n=\sup_n\dim_H C(\tfrac 14-\tfrac 1{5n})=\sup_n\frac{\ln 2}{-\ln(\frac 14-\frac 1{5n})}=\frac 12,$$
and, since $K\setminus K^*$ is countable, also $\dim_HK^*=\frac 12$, which completes the proof.
\end{proof}

Now it is easy to prove the following full ``characterization of smallness'' of sets $L_F$ and $\cv(d_F)$ with
 arbitrary closed $F \subset \R^2$. The subsequent case of {\it compact} sets $F$ (Proposition \ref{Pchardon} and Theorem \ref{chardon}) is much more
 sofisticated.

\begin{theorem}\label{euprouz}
 Let  $A \subset (0, \infty)$. Then the following properties are equivalent.
\begin{enumerate}
\item[(i)]\ $A\subset L_F$ for some closed $F \subset \R^2$.  
\item[(ii)]\ $A\subset \cv(d_F)$ for some closed $F \subset \R^2$.  
\item[(iii)]\ $A$ can be covered by countably many  $\BT$ sets.
\end{enumerate}
 In particular, for each closed $\emptyset \neq F \subset \R^2$, the set $\cv(d_F)$ (and so also $L_F$ and $T_F$) is a countable union of sets with zero $1/2$-dimensional  Minkowski content  and, consequently,
 $\dim_P(\cv(d_F)) \leq 1/2$.
\end{theorem}
\begin{proof}
 By \eqref{tlcv}, (i)$\implies$(ii).
To prove  (ii)$\implies$(iii), first observe that, for any compact $\emptyset \neq F \subset \R^2$,
$$\cv(d_F)=\{0\}\cup\bigcup_{m=1}^\infty (\cv(d_F)\cap[\frac 1m,m]),$$
 and consequently
\begin{equation}\label{cvspocpokr}
\text{$\cv(d_F)$ can be covered by countably many  $\BT$ sets}
\end{equation}
by Theorem \ref{odrn}. For a closed (noncompact) set $\emptyset \neq F$ we get using Lemma \ref{reduk} that
$$ \cv(d_F) \subset  \bigcup_{n=1}^{\infty} \cv(d_{F\cap \overline{B}(0,n)}).$$
 Consequently we obtain that \eqref{cvspocpokr} holds for all closed $\emptyset \neq F\subset \R^2$, and so (ii)$\implies$(iii) follows.

If (iii) holds, then  $A \subset \bigcup_{n=1}^{\infty} K_n$, where each $\emptyset \neq K_n$
 is a $\BT$ set. Since $A\subset\bigcup_m[\frac 1m,\infty)$, we can suppose that $K_n\subset (0,\infty)$, $n\in\N$.  Proposition~\ref{konstr}  implies that there exist a compact sets $\emptyset \neq F_n \subset \R^2$ 
	   such that  $K_{n} \subset L_{F_{n}}$. It is easy to construct inductively points $c_n \in \R^2$,
		 $n=1,2,\dots$, such that
		\begin{equation}\label{diskre}
		\dist(c_k + F_k, c_l + F_l) \geq 4 \max(\diam F_k, \diam F_l) +1,\ \ \ \ k,l \in \N,\ k \neq l.
		\end{equation}
		Then  $F:= \bigcup_{n=1}^{\infty} (c_n + F_n)$ is clearly closed. To prove (i), it is sufficient
		 to show that $A \subset L_F$. To this end, consider an arbitrary $r \in A$  and choose
		 $n \in \N$ with $r \in K_n$. Then $S_r(c_n +F_n)\neq \emptyset$ is not a Lipschitz manifold.
		 Since $r \leq \diam F_n$ by \eqref{cvdiam} and \eqref{tlcv}, we obtain by \eqref{diskre}
		that  $S_r(F) \cap ((c_n +F_n) + B(0, 2 \diam F_n))  = S_r(c_n+F_n)$, which implies $r \in L_F$. So (i), (ii) and (iii) are equivalent.
		
The rest follows by \eqref{nulmin} and \eqref{dimp}.		
		\end{proof}

\begin{lemma}\label{rozsek}
Let  $\delta>0$ and  $K \subset [\delta/2,\delta]$ be a  $\BT$ set.
  Then  there exist $p\in \N$ and compact sets $K_1,\dots, K_p$ such that 
		$$  K = \bigcup_{i=1}^{p} K_i,$$ 
		\begin{equation}\label{vlroz}
	G_{1/2} (K_i) \leq 2 \sqrt{\delta},\ i=1,\dots,p, \ \ \text{and}\ \ \ p \leq G_{1/2}(K)/\sqrt{\delta} +1.
	\end{equation}
\end{lemma}
\begin{proof}
We can clearly suppose that  $G_{1/2} (K) > 2 \sqrt{\delta}$.
Set $g(t):=  G_{1/2} (K\cap [\delta/2,t]), \ t \in K$. It is easy to see that $g$ is increasing
 and continuous on $K$, and consequently, for each $k \in \N$, there exists
$$  t_k:=  \max\{t \in [\delta/2,\delta]\cap K:\ g(t)\leq k \sqrt{\delta}\}.$$
Put $t_0:= \delta/2$ and
 $p:= \min\{ k:\ t_k = \max K\}$. Then clearly  $p \leq G_{1/2}(K)/\sqrt{\delta} +1$.
Set $K_i:= K \cap [t_{i-1}, t_i]$, $i=1,\dots,p$. 
Obviously, it is now sufficient to prove that $G_{1/2} (K_i) \leq 2 \sqrt{\delta},\ i=1,\dots,p$.
  Using Lemma \ref{mongap}, we easily see that this inequality holds if $i\leq 2$. If
 $i>2$, then $g(t_i) \leq i \sqrt{\delta}$ and it is easy to see that $g(t_{i-1}) > (i-2) \sqrt{\delta}$. Therefore   $G_{1/2} (K_i) = g(t_i) - g(t_{i-1}) \leq i \sqrt{\delta} - (i-2) \sqrt{\delta}
 = 2 \sqrt{\delta}$.
\end{proof}

\begin{proposition}\label{Pchardon}
Let $A\subset (0,\infty)$. Then the following conditions are equivalent.
\begin{enumerate}
\item[(i)] $A  \subset  L_F$ for some compact set $F \subset \R^2$.
\item[(ii)] $A  \subset \cv(d_F)$ for some compact set $F \subset \R^2$.
\item[(iii)]  $K: = \overline{A}$ is Lebesgue null and  there exists $D>0$ such that  $K \subset [0, D]$ and 
\begin{equation}\label{rada}
S:= \sum_{n=0}^{\infty}\ (\delta_n)^{3/2} \,  G_{1/2} (K \cap [\delta_{n+1}, \delta_n])  < \infty,
 \ \ \text{where}\ \ \delta_n:= D\,  2^{-n}.
\end{equation}
\end{enumerate}
\smallskip

Moreover, if $K= \cv(d_F)$ for some compact set $F \subset \R^2$ and $D\geq  \diam F>0$, then
  $K \subset [0, D]$ and  \eqref{rada} holds; more precisely, $S \leq 10^{10} D^2$.  
\end{proposition}
\begin{proof}
By \eqref{tlcv}, we know that (i) implies (ii).

To prove that (ii) implies (iii), it is sufficient (due to Lemma \ref{mongap}) to show that
 (iii) holds if $K= \cv(d_F)$ for some compact $F \subset \R^2$. In this case it is well-known
 (see \cite{Fe} or \cite{Fu}) that $K= \cv(d_F)$ is Lebesgue null, and we also have proved it independently,
 see Remark \ref{inde}. Further
  $\cv(d_F) \subset [0,D]$ whenever $D \geq  \diam F$ by  \eqref{cvdiam}. Thus, to prove (iii) and the ``moreover part'' of the assertion, it is sufficient to show that $S \leq 10^{10} D^2$, where $S$ is the sum from \eqref{rada}.
  
 Let $\delta_n$, $H_n$, $P_n \subset H_n$ and $p_n$ 
 be as in Lemma \ref{peter}. Then \eqref{covhn} and \eqref{sumaan} hold.

 For  each  $x \in P_n$, we have  $\delta_n/2 \leq d_F(x) \leq \delta_n$ which gives
  $\overline{B}(x, \delta_n/ 8) \subset  \overline{B}(x, d_F(x)/ 3)$, and consequently
	Lemma \ref{nakouli} yields
	$$ G_{1/2} (d_F( \Crit(d_F) \cap \overline{B}(x, \delta_n/ 8)) \leq 6\cdot 10^4 \sqrt{d_F(x)}
	  \leq 6\cdot 10^4 \sqrt{\delta_n}.$$
		So, using \eqref{covhn}, $\lambda_2(\cv(d_F))=0$, Lemma \ref{aditgap} and  Lemma \ref{mongap}, we obtain
		\begin{multline*}
		G_{1/2} (\cv(d_F) \cap [\delta_{n+1}, \delta_n]) = G_{1/2} (d_F( \Crit(d_F) \cap H_n) \\
		\leq p_n 6 \cdot 10^4 \sqrt{\delta_n} + (p_n-1) \sqrt{\delta_n- \delta_{n+1}} 
		 \leq  10^5 p_n \sqrt{\delta_n}.
		\end{multline*}
	 Hence \eqref{sumaan} implies 
	 $$ S \leq 10^5 \sum_{n=0}^{\infty} p_n (\delta_n)^2 \leq 10^{10} D^2.$$
		
		Now suppose that  condition (iii) holds. Denote $K_n:= K \cap [\delta_{n+1}, \delta_n]$.
		By Lemma \ref{rozsek} we can write
		$$  K_n = \bigcup_{p=1}^{p_n} K_{n,p},$$
		 where all $K_{n,p}$ are compact 
		\begin{equation}\label{rozknp}
	G_{1/2} (K_{n,p}) \leq 2 \sqrt{\delta_n} \ \ \text{and}\ \ \ p_n \leq G_{1/2}(K_n) (\delta_n)^{-1/2} +1.
	\end{equation}
	  Proposition \ref{konstr} and Remark \ref{Fmala} easily imply that there exist closed sets $F_{n,p} \subset B(0, 5 \delta_n)$
	   such that  $K_{n,p} \subset L_{F_{n,k}}$. Using  \eqref{rozknp} we obtain 
	   \begin{multline*}
	    \sum_{p=1}^{p_n}  \lambda_2( B(0, 9 \delta_n)) \leq   ( G_{1/2}(K_n) (\delta_n)^{-1/2} +1)\cdot  81 \pi \cdot (\delta_n)^2 \\
	    =81 \pi\cdot  G_{1/2}(K_n) (\delta_n)^{3/2} + 81 \pi (\delta_n)^2 =: b_n.
	\end{multline*}
	 Using \eqref{rada}, we obtain that  $\sum b_n$ converges and consequently (see, e.g., \cite{Ko}) there exist  $c_{n,p}\in \R^2$ such that
	   the system  $\{ B(c_{n,p},  9 \delta_n): 0\leq n,\, 1\leq p \leq p_n\}$ is disjoint and its union is a bounded set.
	   Then  the set
	   $$ F:= \overline{ \bigcup \{F_{n,p}+ c_{n,p}:\  0\leq n,\, 1\leq p \leq p_n\}}$$
	    is clearly compact. Moreover, it is easy  to show that $K \setminus \{0\} \subset L_F$ and thus obtain (i).
			 Indeed, consider an arbitrary $r \in K \setminus \{0\}$ and choose $n\in \N$ and
			 $1\leq p \leq p_n$ such that $r\in K_{n,p}$. Then $S_r(F_{n,p}+c_{n,p})\neq \emptyset$
			 is not a Lipschitz manifold. Since $r \leq \delta_n$, we obtain that
			$   S_r(F) \cap  B(c_{n,p},  7 \delta_n) =   S_r(F_{n,p}+ c_{n,p})$ which easily implies $r \in L_F$.
			 \end{proof} 
			
			Using Lemma \ref{sumint}, we now easily infer Theorem \ref{chardon} from  Proposition \ref{Pchardon}. In fact,  Theorem \ref{chardon} is factually
			 only a more elegant reformulation of Proposition \ref{Pchardon}, which is, on the other hand,
			 more suitable for applications.
			
\begin{proof}[Proof of Theorem \ref{chardon}]
To infer Theorem \ref{chardon} from  Proposition \ref{Pchardon}, it is clearly sufficient to
 show that Theorem \ref{chardon} (iii) is equivalent to Proposition \ref{Pchardon} (iii).
 Thus it is sufficient to prove that, if $D>0$ and $K \subset [0,D]$ is a compact Lebesgue null set, then
 the conditions
\begin{equation}\label{koneint}
 \int_0^\infty G_{1/2}(K\cap[r,\infty)) \sqrt{r}\, dr<\infty
\end{equation}   
 and
\begin{equation}\label{konesou}
 \sum_{n=0}^{\infty}\ (\delta_n)^{3/2} \,  G_{1/2} (K \cap [\delta_{n+1}, \delta_n])  < \infty,
 \ \ \text{where}\ \ \delta_n:= D\,  2^{-n},
\end{equation}
 are equivalent. To this end, set  $g(r):=  G_{1/2}(K\cap[r,\infty)),\ r>0$. If $g(r)= \infty$
 for some $r>0$, then clearly both \eqref{koneint} and \eqref{konesou} (by Corollary \ref{adbt})
 fail. Otherwise we can use Lemma \ref{sumint} with $\alpha= 3/2$ and obtain that \eqref{koneint} is equivalent
  to the condition
	\begin{equation}\label{trihvezd}   
	 \sum_{k=0}^{\infty}  (g(\delta_{k+1})- g(\delta_k)) \cdot (\delta_k)^{3/2} < \infty.
	\end{equation}
For each $k\geq 0$,
$ g(\delta_{k+1}) \geq  g(\delta_k) + G_{1/2} (K \cap [\delta_{k+1}, \delta_k])$ by the definition 
 of the gap sum and Lemma \ref{aditgap} implies $ g(\delta_{k+1}) \leq  g(\delta_k) + G_{1/2} (K \cap [\delta_{k+1}, \delta_k]) + \sqrt D$. Consequently,
\begin{multline*} G_{1/2} (K \cap [\delta_{k+1}, \delta_k])
 (\delta_k)^{3/2} \leq
(g(\delta_{k+1})- g(\delta_k)) \cdot (\delta_k)^{3/2}\\ \leq G_{1/2} (K \cap [\delta_{k+1}, \delta_k])
 (\delta_k)^{3/2} + \sqrt D \cdot (\delta_k)^{3/2}
\end{multline*}
  and so \eqref{konesou} is equivalent to \eqref{trihvezd}.
\end{proof}			
			
			An interesting consequence of Proposition \ref{Pchardon} follows.

\begin{theorem}\label{45}
 \begin{enumerate}
 \item[(i)]  If $\emptyset \neq F \subset \R^2$ is a compact set, then $\cv(d_F)$ has zero $4/5$-dimensional  Minkowski content, in particular
$\overline{\dim}_M(\cv(d_F)) \leq 4/5$.  
\item[(ii)] There exists a compact set  $F \subset \R^2$ such that $G_{4/5} (\cv(d_F))       
 = \infty$ and so  $\overline{\dim}_M(\cv(d_F)) = 4/5$.
\end{enumerate}
\end{theorem}
\begin{proof}
Set $D:= \diam{F}$, $\delta_n:= D\,  2^{-n}$, $K:= \cv(d_F)$ and $B:= K \cup\{\delta_n:\ n=0,1,\dots\}$. Since clearly $G_{1/2} (B\cap [\delta_{n+1}, \delta_n]) \leq G_{1/2} (K \cap [\delta_{n+1}, \delta_n]) + 2 \sqrt{\delta_n}$, we deduce from \eqref{rada} that
\begin{equation}\label{radaL}
\sum_{n=0}^{\infty} a_n < \infty,
\end{equation}
 where we put  $a_n:= (\delta_n)^{3/2} \,  G_{1/2} (B \cap [\delta_{n+1}, \delta_n]) $.
 Note that $B$ is compact and Lebesgue null.

For $r>0$, we will use the following obvious version of \eqref{expl}:
$$  \lambda_1(B_r)= 2r + \sum_{I \in \cal G_B} \min(2r, |I|).$$
 Consequently, denoting  $\cal G_n:= \cal G_{B\cap [\delta_{n+1}, \delta_n]}$ and
 $\sigma_n=\sigma_n(r):= \sum_{I \in \cal G_n} \min(2r, |I|)$, we obtain  $\lambda_1(B_r)= 2r + \sum_{n=0}^{\infty}
 \sigma_n$ and
\begin{equation}\label{podil}
\frac{\lambda_1(B_r)}{r^{1/5}} = 2 r^{4/5} + \sum_{n=0}^{\infty} \sigma_n r^{-1/5}.
\end{equation}
Obviously,  we have
\begin{equation}\label{prodh}
\sigma_n \leq \lambda_1([\delta_{n+1}, \delta_n]) = \delta_n/2.
\end{equation}
Further, using the inequality  $\min(a,b) \leq \sqrt a\cdot \sqrt{b},\ (a,b>0)$, we obtain
\begin{equation}\label{drodh}
\sigma_n \leq  \sum_{I \in \cal G_n} \sqrt{2r} \sqrt{|I|} =  \sqrt{2r}\cdot G_{1/2} (B \cap [\delta_{n+1}, \delta_n]) = \sqrt{2r}\cdot a_n (\delta_n)^{-3/2}.
\end{equation}

Now fix an arbitrary $\ep>0$. Choose $p\in \N$ such that  $4D 2^{-p} < \ep/4$. 

Further, consider $0<r<1$ and the corresponding $k=k(r)$ such that  $2^{-(k+1)} \leq r < 2^{-k}$.
Write
\begin{align*}
\sum_{n=0}^{\infty} \sigma_n r^{-1/5} &= \sum_{n\geq \lfloor\frac{k}{5}\rfloor+p} \sigma_n r^{-1/5} 
 + \sum_{\lfloor\frac{k}{6}\rfloor<n < \lfloor\frac{k}{5}\rfloor+p} \sigma_n r^{-1/5} + 
\sum_{0 \leq n \leq \lfloor\frac{k}{6}\rfloor} \sigma_n r^{-1/5}\\
&=: A(r) + B(r) + C(r).
\end{align*}
Using \eqref{prodh}, we obtain
\begin{align*}
A(r) &\leq \sum_{n= \lfloor\frac{k}{5}\rfloor+p}^{\infty} \frac{\delta_n}{2} \left(2^{-(k+1)}\right)^{-\frac 15} =
 \sum_{n= \lfloor\frac{k}{5}\rfloor+p}^{\infty}  \frac D2 \cdot 2^{-n +\frac k5 +\frac 15}
= \frac D2\cdot 2^{\frac k5 +\frac 15}\cdot 2 \cdot 2^{-\lfloor\frac{k}{5}\rfloor -p}\\ 
&\leq 4D 2^{-p} < \ep/4.
\end{align*}

Using \eqref{drodh}, we obtain
\begin{equation}\label{pomdr}
\sigma_n r^{-\frac{1}{5}} \leq \sqrt{2r}\cdot a_n (\delta_n)^{-\frac{3}{2}}r^{-\frac{1}{5}} =
 \sqrt{2} D^{-\frac{3}{2}}    a_n 2^{\frac{3n}{2}} r^{\frac{3}{10}}
\leq \sqrt{2} D^{-\frac{3}{2}} a_n \cdot 2^{\frac{3n}{2}-\frac{3k}{10}}.
\end{equation}
and
$$ 
B(r) \leq \sum_{\lfloor\frac{k}{6}\rfloor<n < \lfloor\frac{k}{5}\rfloor+p}  \sqrt{2} D^{-\frac{3}{2}} a_n \cdot 2^{\frac{3n}{2}-\frac{3k}{10}} \leq  \sqrt{2} D^{-\frac{3}{2}}  2^{\frac{3}{2}(1+p)} \sum_{n= \lfloor\frac{k}{6}\rfloor}^{\infty} a_n.
$$
 So, since $\sum_{n=0}^{\infty} a_n$ converges and $k(r)\to \infty,\ r\to 0+$, there exists
 $1>r_1>0$ such that $B(r)< \ep/4$ whenever  $0<r<r_1$.

Using \eqref{pomdr}, we obtain
$$ C(r)\leq \sqrt{2} D^{-\frac{3}{2}}  2^{-\frac{3k}{10}} 
\sum_{n=0}^{\lfloor\frac{k}{6}\rfloor}   a_n \cdot \left(2^{3/2}\right)^n.$$
 Consequently, denoting $M:= \max\{a_n: n\geq 0\}$, we obtain
$$C(r) \leq \sqrt{2} D^{-\frac{3}{2}} M 2^{-\frac{3k}{10}} \cdot 2 \cdot 2^{\frac 32(\frac k6+1)} \leq 
\sqrt{2} D^{-\frac{3}{2}} M \cdot 2 \cdot 2^{\frac 32} \cdot 2^{-\frac{k}{20}}.
$$
So there exists $0<r_0<r_1$ such that $0<r<r_0$ implies  $C(r)< \ep/4$ and $2r^{4/5} < \ep/4$,
 and consequently, by \eqref{podil} and the above estimates of $A(r)$ and $B(r)$, 
$$
\frac{\lambda_1(B_r)}{r^{1/5}} < \frac\ep 4 + \frac\ep 4 +\frac\ep 4 +\frac\ep 4 = \ep.
$$
Thus we have proved that  $\frac{\lambda_1(B_r)}{r^{1/5}} \to 0,\ r\to 0+$; in other words
 $B$ (and consequently also $\cv(d_F)$) has  zero $4/5$-dimensional  Minkowski content.

	To prove (ii), set for each $n\in \N$, $\delta_n:= 2^{-n}$ and $k_n:= \lfloor(n\delta_n)^{-4}\rfloor+1$. Clearly
	 \begin{equation}\label{odhk}
	  (n\delta_n)^{-4} \leq k_n \leq 2 (n\delta_n)^{-4},\ \ n \in \N.
	  \end{equation}
	  For each $n \in \N$, let $\delta_n/2 =t_0^n <t_1^n<\dots < t_{k_n}^n=\delta_n$ be the equidistant partition of
 the interval $[\delta_n/2, \delta_n]$. Set
   $K:=  \{t^n_i:\ n\in \N, 0\leq i \leq k_n\} \cup \{0\}.$
   We will show that $K$ satisfies condition (ii) of Theorem \ref{chardon} and so there exists a compact set  $F \subset \R^2$ such that $K \subset \cv(d_F)$. To this end set $D:=1$.  Using \eqref{odhk}, we obtain
   $$ (\delta_n)^{3/2} \,  G_{1/2} (K \cap [\delta_{n+1}, \delta_n]) =  (\delta_n)^{3/2} \, k_n \sqrt{\frac{\delta_n}{2 k_n}}
    \leq  (n \delta_n)^{-2} (\delta_n)^2 = n^{-2},$$
     and consequently \eqref{rada} holds.
     
    Further, using \eqref{odhk}, we obtain 
 $$  G_{4/5} (K \cap [\delta_{n+1}, \delta_n]) = k_n \left(\frac{\delta_n}{2 k_n}\right)^{4/5} \geq
 \frac 12 (k_n)^{1/5} (\delta_n)^{4/5} \geq \frac 12 n^{-4/5},$$
 and consequently
$$  G_{4/5} (K) = \sum_{n=1}^{\infty}  G_{4/5} (K \cap [\delta_{n+1}, \delta_n]) = \infty.$$
 Now (ii) follows by Lemma \ref{mongap} and \eqref{hawk}.
\end{proof}

We present also the following slightly curious consequence of Proposition \ref{Pchardon}
 which clearly cannot be obtained from any ``measure smallness'' of $\cv(d_F)$.

\begin{proposition}\label{konmn}
Let $D>0$ and $[\alpha, \beta] \subset (0,\infty)$. Then there exists a finite set
 $P \subset [\alpha, \beta]$ of cardinality $p:=\lfloor 10^{25} D^4 (\beta-\alpha)^{-1} \beta^{-3/2}\rfloor +3 $\ such that \ $P \setminus \cv(d_F) \neq \emptyset$ whenever
 $F \subset \R^2$ is compact and $\diam F \leq D$.
	\end{proposition}
\begin{proof}
Set $\delta_n:= D 2^{-n}$, $n=0,1,\dots$.
 We can suppose that $\beta \leq D$; otherwise the assertion is trivial by \eqref{cvdiam}.
 Then $\beta \in (\delta_{k+1}, \delta_k]$ for some $k$.
Set
$$n:=\begin{cases}
k &\text{ if } \beta-\delta_{k+1}\geq\frac 14(\beta-\alpha),\\
k+1 &\text{ otherwise.}
\end{cases}$$
Then, for $[c,d]:=[\alpha,\beta]\cap[\delta_{n+1},\delta_n]$, we have
\begin{equation} \label{asppul}
d-c\geq \frac{\beta-\alpha}{4}\,\text{ and }\,\delta_n\geq\frac{\beta}{2}.
\end{equation}
Let $c = t_0< t_1< \dots < t_{p-1}= d$ be the equidistant partition of
 the interval $[c,d]$ and $P:=\{t_0,\dots,t_{p-1}\}$. Suppose to the contrary
 that $F \subset \R^2$ is compact, $\diam F \leq D$ and 
$P \subset \cv(d_F)$. Then we have, by the ``moreover part'' of Proposition
 \ref{Pchardon} and \eqref{asppul},
\begin{multline*}
 10^{10} D^2 \geq G_{1/2} (\cv(d_F) \cap [\delta_{n+1}, \delta_n])
 (\delta_n)^{3/2} \geq   G_{1/2}(P) (\beta/2)^{3/2}\\
= (p-1) \left(\frac{d-c}{p-1}\right)^{1/2} (\beta/2)^{3/2} \geq (p-1)^{1/2} \left(\frac{\beta-\alpha}{4}\right)^{1/2}(\beta/2)^{3/2}\\ \geq  8^{-1}(p-1)^{1/2} (\beta-\alpha)^{1/2} \beta^{3/2}, 
\end{multline*}
 which, as an elementary computation shows, contradicts the choice of $p$.
\end{proof}
\begin{remark}\label{porinn}
A similar argument gives the following result
saying that $\cv(d_F)$ is porous at $0$ in a (rather weak) sense.
\medskip

{\it Let $F \subset \R^d$ be compact. For each $r>0$, let $\gamma_r(F)$ be the lenght of the largest
 component of $(0,r)\setminus \cv(d_F)$. Then
$$ \liminf_{r\to 0+}  \frac{\gamma_r(F)}{r^5} >0.$$ }
\end{remark}

\section{Results in Riemannian manifolds}
\textsl{}
\newcommand{\D}{\mathrm{d}}

Let $M$ be a two-dimensional smooth, connected and complete Riemannian manifold, and let $\dist(\cdot,\cdot)$ denote the intrinsic distance in $M$. As a consequence of the Hopf-Rinow theorem (see \cite{Petersen}*{Theorem~5.7.1}), $M$ is boundedly compact, i.e., 
\begin{equation} \label{bcom}
\text{every bounded and closed subset of }M\text{ is compact.}
\end{equation} 

Let $(g_p:\, p\in M)$ be the Riemannian metric on $M$ and let $(U,\varphi)$ be a chart of $M$. The induced metric in $\varphi(U)\subset\R^2$,
$$\tilde{g}_x(\cdot,\cdot):= g_{\varphi^{-1}(x)}\left(D\varphi^{-1}(x)(\cdot),D\varphi^{-1}(x)(\cdot)\right),\quad x\in\varphi(U),$$
where $D\varphi^{-1}(x):\R^2\to T_{\varphi^{-1}(x)}M$ is the differential of $\varphi^{-1}$ at $x\in\varphi(U)$,
is smooth and, hence, given any compact subset $K\subset U$ there exist constants $0<c<C<\infty$ such that
\begin{equation}  \label{Riem_m}
c^2\|u\|^2\leq\tilde{g}_x(u,u)\leq C^2\|u\|^2,\quad u\in\R^2,\, x\in\varphi(K).
\end{equation}

Any chart of $M$ is locally bi-Lipschitz. This is a well-known fact. We present a proof since we could not find a direct reference (although, the proof is implicitly contained in many textbooks, as e.g.\ \cite{Petersen}*{proof of Theorem~5.3.8}).

\begin{fact}  \label{F_bilip}
For any chart $(U,\varphi)$ of $M$ and $p_0\in U$ there exists an open neighbourhood $V\subset U$ of $p_0$ such that $\varphi|V$ is bi-Lipschitz.
\end{fact}

\begin{proof}
Choose $r,\rho>0$ such that $K:=\overline{B}(p_0,r)\subset U$ (note that $K$ is compact by \eqref{bcom}) and $B:=B(\varphi(p_0),\rho)\subset\varphi(K)$. 
\eqref{Riem_m} implies that for any curve $\gamma:[a,b]\to K$, its length $\ell(\gamma)$ is related to the Euclidean length $\ell_e(\tilde{\gamma})$ of its image $\tilde{\gamma}:=\varphi\circ\gamma$ by
$$c\,\ell_e(\tilde{\gamma})\leq\ell(\gamma)\leq C\,\ell_e(\tilde{\gamma}).$$
Choosing for $\gamma$ a minimizing geodesic path connecting two points $p,q\in B(p_0,\frac r3)$ (which must clearly be contained in $K$), we obtain $c\|\varphi(p)-\varphi(q)\|\leq\dist(p,q)$. On the other hand, if $p,q\in V:=B(p_0,\frac r3)\cap\varphi^{-1}(B)$ and the curve $\gamma$ is chosen such that $\tilde{\gamma}$ is the straight segment connecting $\varphi(p)$ and $\varphi(q)$ (in $B\subset\R^2$), we get $\dist(p,q)\leq C \|\varphi(p)-\varphi(q)\|$. Thus, $\varphi$ is bi-Lipschitz on $V$.
\end{proof}

\begin{definition}  \rm
Let $G\subset M$ be open. A function $f:G\to\R$ is \emph{locally semiconcave} if for each chart $(U,\varphi)$ with $U\subset G$, $f\circ\varphi^{-1}$ is locally semiconcave.
\end{definition}

\begin{remark} \rm
\begin{enumerate}
\item Since the change of coordinates is always smooth, and composition with smooth mappings preserves semiconcavity (see \cite{MM}*{Prop.~2.6}), it is enough to check the above properties on any family of charts covering $G$.
\item It is clear from the definition that the above properties are independent on the particular Riemannian metric, they depend only on the differentiability structure of $M$.
\item For a locally semiconvex function $f:G\to\R$ from the above definition, the (one-sided) directional derivative $\D f(p;v)$ exists for any $p\in G$ and $v\in T_pM$ (since the Euclidean case is obvious).
\end{enumerate}
\end{remark}

\begin{definition} \label{D_crit} \rm
Let $G\subset M$ be open and let $f:G\to\R$ be locally Lipschitz. A point $p\in G$ is \emph{critical} for $f$ if for each chart $(U,\varphi)$ with $p\in U$, $\varphi(p)$ is a critical point for $f\circ\varphi$, i.e., if
$$0\in\partial^C(f\circ\varphi^{-1})(\varphi(p)).$$
\end{definition}

\begin{remark} \label{R_crit}\rm
Again, the criticality of a point can be equivalently verified on any chart around $p$. Indeed, if $(U,\varphi)$ and $(V,\psi)$ are two such charts, we have 
by the chain rule for the Clarke subdifferential (see \cite{Cl}*{Theorem~2.3.10})
$$\partial^C(f\circ\psi^{-1})(\psi(p))\subset \partial^C(f\circ\varphi^{-1})(\varphi(p))\circ D(\varphi\circ\psi^{-1})(\psi(p)),$$
and since the differential $D(\varphi\circ\psi^{-1})(\psi(p))$ is regular,  $0\in\partial^C(f\circ\varphi^{-1})(\varphi(p))$ whenever $0\in \partial^C(f\circ\psi^{-1})(\psi(p))$. The other implication follows by symmetry.
\end{remark}

Let $F\subset M$ be a proper nonempty closed subset. The distance function 
$$d_F: p\mapsto \dist(p,F),\quad p\in M\setminus F,$$
is locally semiconcave (see \cite{MM}). For any $p\in M\setminus F$ there exists a unit-speed geodesic path $\gamma:[0,r]\to M$ with $\gamma(0)=p,\gamma(r)\in F$, $r=\dist(p,F)$. We call such a path an \emph{$F$-segment emanating from} $p$. 

Of course, there can exist more $F$-segments emanating from $p$. Given a unit tangent vector $v\in T_pM$, let $\alpha(v)\geq 0$ be the minimal angle $\angle(v,\gamma'(0))$ formed by $v$ and an $F$-segment emanating from $p$ (for the existence of the minimum we use the compactness of $\{\gamma'(0):\,\gamma\text{ is an }F\text{-segment emanating from }p\}$, see \cite{Petersen}*{Prop.~12.1.2~(1)}). The following property is often used as definition of critical points of $d_F$:
\begin{equation} \label{Grove}
p\text{ is critical for }d_F\text{ if and only if }\alpha(v)\leq\tfrac\pi 2\text{ for any unit vector }v\in T_pM.
\end{equation}
Grove \cite{Grove}*{p.~360} claims that a simple argument based on standard local distance comparison shows that the definition \eqref{Grove} is equivalent to another definition (based on \cite{Grove}*{Eq.~(1.1)}), which can be shown to be equivalent to our Definition~\ref{D_crit} using Lemma~\ref{charcrit} and the local semiconcavity of $d_F$ on $M\setminus F$.

The equivalence of Definition~\ref{D_crit} for $f:=d_F$ and \eqref{Grove} can also be derived from the following property of the directional derivatives of $d_F$:
\begin{equation}  \label{cosinus}
\D(d_F)(p,v)=-\cos\alpha(v).
\end{equation}
Equation \eqref{cosinus} is well-known, though it is difficult to find a direct proof. In the setting of Alexandrov spaces with curvature bounded from below, it can be found in \cite{BGP92}*{Example~11.4} (cf.\ also \cite{Plaut}*{p.~880} or \cite{BBI01}*{Exercise~4.5.11, Remark~4.5.12}). When applying to a (noncompact) complete Riemannian manifold, an additional localization argument has to be applied (similarly as in the proof of Proposition~\ref{P1} below). A direct proof in the Riemannian setting using the first variation formula is also possible, using an argumentation similar to the proof of \cite{Petersen}*{Prop.~12.1.2~(4)}.

The following inequality is a version of Ferry's inequality \cite{Fe}*{Proposition~1.5} in Riemannian manifolds.

\begin{proposition}  \label{P1}
Let $M$ be a complete, connected two-dimensional Riemannian manifold, let $F\subset M$ be its closed nonempty proper subset and $\emptyset\neq K\subset M\setminus F$ a compact set. Then there exists $C>0$ such that for any points $x,y\in K$ that are critical for $d_F$, we have
$$|d_F(x)^2-d_F(y)^2|\leq C\: \dist(x,y)^2.$$
\end{proposition}

\begin{proof}
Choose a point $p\in K$ and set $R:=\diam K+\sup_{q\in K}\dist(q,F)$. Using a partition of unity, we can find a ($C^\infty$) smooth mapping $\varphi:M\to[0,1]$ such that $\varphi(q)=1$ if $\dist(p,q)\leq 2R$ and $\varphi(q)=0$ if $\dist(p,q)>3R$. The closed ball $\overline{B}(p,3R)$ is compact in $M$ (see \eqref{bcom}), hence, the sectional curvature is bounded on $B(p,3R)$. By \cite{Greene}*{Corollary~of~Theorem~1}, there exists a complete Riemannian metric $\hat{g}=(\hat{g}_q)_{q\in M}$ on $M$ with bounded sectional curvature. Then, if $g$ denotes the original Riemannian metric on $M$, $\tilde{g}_q:=\varphi(q)g_q+(1-\varphi(q))\hat{g}_q$, $q\in M$, defines a new Riemannan metric on $M$ which has bounded sectional curvature and agrees with $g$ on $B(p,2R)$. The distances and geodesic segments remain the same in the new metric for any two points lying in $B(p,R)$. Also, the new metric $\tilde{g}$ is complete, since it coincides with the complete metric $\hat{g}$ outside of the ball $\overline{B}(p,3R)$.

Let $x,y\in (M\setminus F)\cap K$ be critical points of $d_F$ and let $\gamma_{xy}$ be a minimizing geodesic path connecting $x$ and $y$ with $\gamma_{xy}(0)=x$. By \eqref{Grove}, there exists an $F$-segment $\gamma_{xz}$ emanating from $x$, ending in some point $z\in F$ and such that 
$\alpha:=\angle(\gamma_{xy}'(0),\gamma_{xz}'(0))\leq\frac\pi 2$.
Let $\gamma_{yz}$ be a minimizing geodesic path connecting $y$ and $z$. Note that all the three points $x,y,z$ lie in $B(p,R)$.

Let $k<0$ be a lower sectional curvature bound of $(M,\tilde{g})$ and let $H_k^2$ denote the hyperbolic plane with constant curvature $k$. Let $\tilde{x},\tilde{y},\tilde{z}$ be vertices of a ``comparison triangle'' of $x,y,z$ in $H_k^2$ (i.e., the lengths of edges connecting the corresponding points are the same). Then, by the Toponogov theorem (see \cite{ChE}*{Theorem~2.2}), the angle 
$$\tilde{\gamma}:=\angle(\tilde{y},\tilde{x},\tilde{z})\leq\alpha\leq\frac\pi 2.$$  
Denote $a:=\dist(x,z)=\dist(\tilde{x},\tilde{z})=d_F(x)$, $b:=\dist(x,y)=\dist(\tilde{x},\tilde{y})$, $c:=\dist(y,z)=\dist(\tilde{y},\tilde{z})$. Note that $a,b$ and $c$ are majorized by $2R$.
The hyperbolic law of cosines in $H_k^2$ with $\kappa:=\sqrt{-k^{-1}}$ (see \cite{BH}*{I.2.13,~p.24}) yields
$$
\cosh \frac c\kappa =\cosh \frac a\kappa\cosh \frac b\kappa - \sinh\frac a\kappa\sinh\frac b\kappa\cos\tilde{\gamma}
\leq\cosh \frac a\kappa\cosh \frac b\kappa.
$$
Using the Taylor expansion one can easily show that
\begin{equation} \label{cosh}
\cosh u-1\leq\frac{u^2}2\cosh u,\quad |u^2-v^2|\leq 2|\cosh u-\cosh v|,\quad u,v\in\R,
\end{equation}
hence
$$\cosh\frac c\kappa-\cosh\frac a\kappa\leq \left(\cosh\frac b\kappa -1\right) \cosh\frac a\kappa\leq\frac{b^2}{2\kappa^2} \cosh\frac b\kappa \cosh\frac a\kappa\leq \frac{b^2}{2\kappa^2}\cosh^2\frac{2R}\kappa.$$
From the fact that $d_F(y)\leq c$ it follows that
$$\cosh\frac{d_F(y)}{\kappa}-\cosh\frac{d_F(x)}{\kappa}\leq \frac{\dist(x,y)^2}{2\kappa^2}\cosh^2\frac{2R}\kappa,$$
and since $x$ and $y$ can be interchanged, we can even add the absolute value to the left hand side.
Using the second inequality from \eqref{cosh}, we get
$$|d_F(y)^2-d_F(x)^2|\leq C\:\dist(x,y)^2$$
with $C=\cosh^2\frac{2R}\kappa$. 
\end{proof}

Using the estimate $|d_F(x)-d_F(y)|\leq |d_F(x)^2-d_F(y)^2|/|d_F(x)+d_F(y)|$, we obtain

\begin{corollary}  \label{C_Ferry}
Let $M,F$ be as in Proposition~\ref{P1} and $z\in M\setminus F$. Then there exist a neighbourhood $V$ of $z$ in $M\setminus F$ and $D>0$ such that for any points $x,y\in V$ that are critical for $d_F$, we have
$$|d_F(x)-d_F(y)|\leq D\: \dist(x,y)^2.$$
\end{corollary}

Using the above inequality, we are able to obtain a weaker (non-quantitative) version of Lemma~\ref{nakouli} in the Riemannian setting.

\begin{lemma}\label{posc}
Let $M$ be a connected complete two-dimensional Riemannian manifold. Let $\emptyset \neq F \subset M$ be a closed set and $z \in M \setminus F$. Then there exists $r>0$ such that
 $d_F(\Crit(d_F) \cap \overline B(z,r))$ is a $\BT$ set.   
\end{lemma}
\begin{proof}
By Corollary~\ref{C_Ferry}, there exist $\delta, D>0$ such that
\begin{equation}\label{fene}
 |d_F(x) - d_F(y)| \leq D\: \dist(x,y)^2\text{ whenever }x,y \in B(z,\delta) \cap \Crit(d_F).
 \end{equation}
Choose a chart $(V, \varphi)$ and an open set $U\subset V\setminus F$ such that $z\in U$ and $\overline{U}\subset V\subset B(z,\delta)$. Due to Fact~\ref{F_bilip} we can suppose that $\vf:U \to \R^2$
 is bilipschitz; choose $\eta>0$ such that $\vf^{-1}$ is $\eta$-Lipschitz.
Denote $z^*:=\vf(z)$,  $U^*:= \vf(U)$ and $d^* := d_F \circ \vf^{-1}$.
As already noted, $d_F$ is locally semiconcave (see \cite{MM}) and, consequently, $d^*$ is locally semiconcave and thus also locally DC on $U^*$. Observe that if $x\in U$, $x^*:=\vf(x)$, $\gamma$ is an $F$-segment emanating from $x$ and $v^*:=D\varphi(x)(\gamma'(0))$ then
$$(d^*)'_+(x^*,v^*)=\D (d_F)(x,\gamma'(0))=-1.$$
Due to \eqref{Riem_m} there exists $c>0$ (independent of $x\in U$) such that 
$$\|v^*\|^2\leq c^{-2}g_x(\gamma'(0),\gamma'(0))=c^{-2}$$ 
and, hence, $(d^*)'_+(x^*,v^*/\|v^*\|)\leq -c$.
By Lemma \ref{pokr} there exist $\rho>0$ and
 Lipschitz graphs $P_1,\dots, P_s$ in $\R^2$
 such that $B(z^*,\rho) \subset U^*$ and  $\Crit(d^*) \cap B(z^*,\rho) \subset  (P_1\cup\dots\cup P_s)$.
 
Choose $r>0$ such that $B:= \overline{B}(z,r)\subset \vf^{-1}(B(z^*, \rho))$, and set $B^*:= \vf(B)$. 
By Definition~\ref{D_crit}, $\Crit (d_F) \cap B = \vf^{-1}(\Crit(d^*) \cap B^*)$ and, therefore,
\begin{equation}\label{rozls}
d_F(\Crit(d_F) \cap B) = d^*(\Crit(d^*) \cap B^*) 
= \bigcup_{i=1}^s d^*(\Crit(d^*) \cap B^* \cap P_i). 
\end{equation}
Now fix $1\leq i \leq s$ and set $S:=P_i$, $P:= \Crit(d^*) \cap P_i \cap B^*$ and $\kappa:= d^*\restriction_P$. Then $P\subset S$ is compact
 and $\diam P \leq 2 \rho$. 
 For each $p_1,\, p_2 \in P$ we have  $\vf^{-1}(p_1), \vf^{-1}(p_2) \in B(z,\delta) \cap \Crit(d_F)$ and, hence, \eqref{fene} yields
  \begin{multline*}
   |d^*(p_1) - d^*(p_2)| = |d_F(\vf^{-1}(p_1)) - d_F(\vf^{-1}(p_2))|\\
    \leq D \|\vf^{-1}(p_1)- \vf^{-1}(p_2)\|^2
   \leq D \eta^2 \|p_1-p_2\|^2.
 \end{multline*}
Applying now Lemma \ref{obra2} to $\kappa$, $P$, $S$ and $C:=D \eta^2$, we obtain that $\kappa(P)=d^*(P)$ is a $\BT$ set. The assertion of the lemma follows by \eqref{rozls} and Corollary~\ref{adbt}.
\end{proof}

As in the Euclidean case, if $\emptyset\neq F\subset M$ is a closed subset, we write $L_F$ ($T_F$) for the set of all
 $r>0$ for which $S_r=\{x\in M:\, d_F(x)=r\}$ is nonempty and it is not a topological (Lipschitz, resp.) $1$-dimensional manifold. Also, the inclusions
\begin{equation} \label{tlcv_riem}
T_F\subset L_F\subset \cv(d_F)
\end{equation}
hold. The first one is obvious, and the second one can be seen as follows. If $r>0$ is not a critical  value of $d_F$ and $(U,\vf)$ a bi-Lipschitz chart of $M$ (cf.\ Fact~\ref{F_bilip}) then $r$ is not a critical value of $d^*:=d_F\circ\vf^{-1}$ and the Clarke's implicit function theorem (cf. \cite{Fu}*{Theorem~3.1}) implies that $(d^*)^{-1}(r)=\vf(S_r\cap U)$ is empty or a $1$-dimensional Lipschitz manifold and, hence, so is $S_r\cap U$. Thus $r\not\in L_F$.

We are now able to prove Theorem~\ref{proko}. 

\begin{proof}[Proof of Theorem~\ref{proko}]
For each $z \in M \setminus F$, we choose by Lemma \ref{posc} $r(z)>0$ such that
 $d_F(\Crit(d_F) \cap \overline B (z,r(z)))$ is a $\BT$ set. The set $C:= (d_F)^{-1}([\ep,K])$ is clearly bounded and closed, and consequently compact (see \eqref{bcom}). Consequently we can choose points
 $z_1,\dots, z_p \in C$ such that  $C \subset  \bigcup_{i=1}^p  \overline B (z_i,r(z_i))$.
 Therefore
$$ d_F(\Crit(d_F) \cap C) \subset \bigcup_{i=1}^p  d_F(\Crit(d_F) \cap \overline B (z_i,r(z_i))).$$
Since  $\cv (d_F) \cap [\ep,K] = d_F(\Crit(d_F) \cap C)$ is compact, $\cv (d_F) \cap [\ep,K]$
is a $\BT$ set by  Corollary \ref{adbt}.
\end{proof}

A corollary for general closed  sets is rather straightforward.

\begin{theorem}\label{prouz}
 If $X$ is a connected complete two-dimensional Riemannian manifold  and $\emptyset \neq F\subset X$ a closed set,
 then both $\cv(d_F)$ and $L_F$ can be covered by countably many  $\BT$ sets. In particular, 
 $\cv(d_F)$ is a countable union of sets of zero $1/2$-dimensional Minkowski content
 and so $\dim_P(\cv(d_F)) \leq 1/2$.
\end{theorem}

\begin{proof}
Assume first that $F$ is compact. Then the assertion follows from Theorem \ref{proko} since
$$\cv(d_F)=\{0\}\cup\bigcup_{m=1}^\infty (\cv(d_F)\cap[\frac 1m,m]).$$
For a closed (noncompact) set $F$ we get using Lemma \ref{reduk} that, for any fixed $a\in X$,
$$ \cv(d_F) \subset  \bigcup_{n=1}^{\infty} \cv(d_{F\cap \overline{B}(a,n)}),$$
and since any bounded and closed subset of $M$ is compact (cf.\ \eqref{bcom}), the proof of the main statement is finished.

The rest follows by \eqref{nuldim} and \eqref{dimp}.
\end{proof}

\begin{remark}\label{inde2}
 We have proved (see \eqref{nuldim}), by different arguments, the result of \cite{RZ1} that
$\cal \cal H^{1/2}(\cv(d_F))=0$. 
\end{remark}

\end{document}